\documentclass[11 pt, leqno]{amsart}

\usepackage {amsfonts}
\usepackage{amsthm}
\usepackage{amssymb}
\usepackage{latexsym}
\usepackage{amsmath}
\usepackage{xcolor}

\usepackage{a4wide}
\usepackage{comment}

\theoremstyle{plain}
\newtheorem{tw}{Theorem}[section]

\newtheorem {lem} [tw]{Lemma}
\newtheorem {prop}[tw] {Proposition}

\newtheorem{cor}[tw]{Corollary}

\theoremstyle{definition}
\newtheorem {deft}[tw] {Definition}
\newtheorem {rem} [tw]{Remark}

\DeclareMathAlphabet{\eufrak}{U}{}{}{}  % Euler fraktur math
\SetMathAlphabet\eufrak{normal}{U}{euf}{m}{n}
\SetMathAlphabet\eufrak{bold}{U}{euf}{b}{n}

\newcommand{\bc} {\Bbb C}
\newcommand{\bn}{\Bbb N}
\newcommand{\br}{\Bbb R}

\newcommand{\Com}{\Delta}
\newcommand{\Cou}{\epsilon}
\newcommand{\wot}{\overline{\ot}}
\newcommand{\ComGG}{\Delta^{\iota, \cop}}

\newcommand{\alg} {\mathsf{A}}
\newcommand {\Tr} {{\textrm{Tr}}}

\newcommand {\Pol} {\tu{Pol}}
\newcommand {\PolQG} {\tu{Pol}(\QG)}

\newcommand{\tu}{\textup}

\newcommand{\mlg}{\mathsf{M}}
\newcommand{\Hil}{\mathsf{H}}

\newcommand{\Alg}{\mathcal{A}}

\newcommand{\cop}{\tu{cop}}
\newcommand{\op}{\tu{op}}
\newcommand{\Irr}{\tu{Irr}}

\newcommand{\QG}{\mathbb{G}}
\newcommand{\QH}{\mathbb{H}}

\newcommand{\Comp}{\mathbb{C}}

\newcommand{\F}{\mathcal{F}}

\newcommand{\Gb}{\mathbb{G}}
\newcommand{\Hb}{\mathbb{H}}

\newcommand{\M}{\mathcal{M}}
\newcommand{\n}{\mathbb{N}}

\newcommand{\Tb}{\mathbb{T}}

\newcommand{\z}{\mathbb{Z}}

\newcommand{\prt}{\widehat{\otimes}}

\newenvironment{rlist}
{

\begin{enumerate}}
{\end{enumerate}}

\newcommand{\la}{\langle}
\newcommand{\ra}{\rangle}
\newcommand{\id}{\textup{id}}

\DeclareMathOperator{\C}{C}
\DeclareMathOperator{\A}{A}

\DeclareMathOperator{\Lt}{L}

\newcommand{\cst}{\ifmmode\mathrm{C}^*\else{$\mathrm{C}^*$}\fi}
\DeclareMathOperator{\Linf}{L^\infty\!\!\;}

\newcommand{\hh}[1]{\widehat{#1}}

\newcommand{\hQG}{\hh{\QG}}

\newcommand{\hQH}{\hh{\QH}}

\newcommand{\ot}{\otimes}

\newcommand{\wt}{\widetilde}
\newcommand{\ltwo}{\Lt^2}

\numberwithin{equation}{section}


\begin{document}

%\usetikzlibrary{arrows,positioning}

\author{Uwe Franz}

\address{%
D\'epartement de math\'ematiques de Besan\c{c}on \\
Universit\'e de Franche-Comt\'e \\
16, route de Gray \\
25 030 Besan\c{c}on cedex, France}

\email{uwe.franz@univ-fcomte.fr}

\author{Hun Hee Lee}

\address{%
Department of Mathematical Sciences  and Research Institute of Mathematics, Seoul National University,
Gwanak-ro 1, Gwanak-gu, Seoul 08826, Republic of Korea}

\email{hunheelee@snu.ac.kr}

\author{Adam Skalski}

\address{
Institute of Mathematics of the Polish Academy of Sciences \\
ul.~\'Sniadeckich 8, 00--656 Warszawa, Poland}

\email{a.skalski@impan.pl}

\title{\bf Integration over the quantum diagonal subgroup and associated Fourier-like algebras}

\begin{abstract}
By analogy with the classical construction due to Forrest, Samei and Spronk we associate to every compact quantum group $\QG$ a completely contractive Banach algebra $\A_\Delta(\QG)$, which can be viewed as a deformed Fourier algebra of $\QG$. To motivate the construction we first analyse in detail the quantum version of the integration over the diagonal subgroup, showing that although the quantum diagonal subgroups in fact never exist, as noted earlier by Kasprzak and So\l tan, the corresponding integration represented by a certain idempotent state on $\C(\QG)$ makes sense as long as $\QG$ is of Kac type. Finally we analyse as an explicit example the algebras $\A_\Delta(O_N^+)$, $N\ge 2$, associated to Wang's free orthogonal groups, and show that they are not operator weakly amenable.
\end{abstract}

\maketitle

Let $G$ be a compact group. Since the celebrated work of Barry Johnson in \cite{BJohns} it has become clear that the Fourier algebra $\A(G)$ may exhibit some interesting Banach algebraic properties circling around the concept of amenability and related to the representation theory of $G$. Later years brought several developments of the ideas of \cite{BJohns}, leading in various directions. On one hand, the Fourier algebras started to be viewed as Banach algebras equipped with a natural, compatible operator space structure, so that the concepts such as operator amenability started playing a more prominent role (\cite{Ruan}). On the other, there have been several contributions replacing $\A(G)$ by other (operator) Banach algebras arising naturally in the study of compact groups. One of such examples was studied in \cite{FSS}, where Brian Forrest, Ebrahim Samei and Nico Spronk investigated Fourier-like algebras `of functions constant on cosets', of which a particular case was $\A_\Delta(G)$, arising from looking at the diagonal subgroup of $G \times G$. Forrest, Samei and Spronk showed that one can view  $\A_\Delta(G)$ as an algebra related to the twisted convolution on $G$ and computed explicitly its norm, which later allowed the authors of \cite{LSS} to interpret  $\A_\Delta(G)$ as a particular instance of the whole family of `$p$-Fourier algebras'. The latter turn out to have interesting amenability properties, in a sense bringing us back to the original motivations of Johnson.

It is thus natural to consider similar concepts for compact quantum groups in the sense of Woronowicz (\cite{worLH}). These are abstract objects, studied in terms of their `algebras of functions'. This makes it possible to transfer several natural notions of the classical theory to the quantum context -- so that, for example, we can consider the Fourier algebra $\A(\QG)$ for a compact quantum group $\QG$, or a notion of a quantum subgroup of a given group  -- but also makes others behave in a rather unnatural way. For example, non-classical groups do not admit quantum diagonal subgroups, as shown in \cite{KShom}. This seems at first glance to close the door to defining a version of the Fourier-like algebra $\A_\Delta(G)$ in the quantum world. A closer look shows, however, that the construction of Forrest, Samei and Spronk uses in fact only the idea of \emph{integrating over the diagonal subgroup}. We show here that if a compact quantum group $\QG$ is \emph{of Kac type}, then the integration as above has a satisfactory equivalent, represented by a particular idempotent state on $\QG \times \QG$; this brings us, on one hand, towards the theory of idempotent states as developed for example in \cite{FST} and, on the other, to the theory of quantum homogeneous spaces, going back to the work of Podle\'s in \cite{Podles}. Further it leads to a natural emergence of twisted convolution of \cite{FSS} and motivates a natural definition of the Fourier-like algebra $\A_\Delta(\QG)$ for any compact quantum group $\QG$ of Kac type. The definition itself suggests a possible extension beyond the Kac case, where some of the original motivations cease to exist. In special examples the arising objects can be described quite explicitly, and we finish the paper with studying such instances. In particular, we obtain a non-operator weakly amenable algebra $\A_\Delta(O_N^+)$, $N\ge 2$, where $O_N^+$ denotes the free orthogonal group of Wang (\cite{Wangfree}). This brings us back once again to the topics studied in the inspirational article \cite{BJohns}.

The detailed plan of the paper is as follows: in Section 1 we present the quantum group background and recall the classical construction of $\A_\Delta(G)$ from \cite{FSS}. Short Section 2 extends basic equivalences concerning the idempotent states, shown in \cite{FS} and \cite{FST} for coamenable compact quantum groups, to the general case (where they need to be slightly reformulated). In Section 3 we introduce the idempotent state on $\C(\QG \times \QG)$ corresponding to the integration over the diagonal subgroup, connect it to the considerations in \cite{KShom} and discuss a hierarchy of properties related to diagonal subgroups in the classical and quantum context. Section 4 defines the Fourier-like (completely contractive) Banach algebra $\A_\Delta(\QG)$ for a compact quantum group $\QG$ of Kac type and mentions possible generalizations beyond the Kac case. Finally in Section 5 we discuss a general strategy of showing that a given $\A_\Delta(\QG)$ is not operator weakly amenable and apply it in the concrete case of free orthogonal groups $\QG=O_N^+$, $N\ge 2$. We end this paper with remarks on the case of the `standard' Fourier algebra $\A(\QG)$ focusing on its non-weak amenability including the case of $\QG = O_N^+$, $N\ge 2$.

\vspace*{0.5 cm}
\noindent
    {\bf Acknowledgement.}\
AS  was partially supported by the NCN (National Centre of Science) grant
2014/14/E/ST1/00525. UF and AS acknowledge support by the French MAEDI and MENESR and by the Polish MNiSW through the Polonium programme. HHL and UF acknowledge support by the French MAEDI and MENESR and the Korean NRF through the Star programme. HHL was partially supported by Promising-Pioneering Researcher Program through Seoul National University(SNU) in 2015 and the Basic Science Research Program through the National Research Foundation of Korea (NRF), grant  NRF-2015R1A2A2A01006882.

\section{Notations and background} \label{Notations}

The algebraic tensor product will be usually denoted by $\odot$ with $\ot$ reserved for the tensor product of maps and minimal/spatial tensor product of \cst-algebras  or operator spaces and $\wot$ for the ultraweak tensor product of von Neumann algebras. For a positive integer $n$ by $M_n$ we mean the space of all complex $n$ by $n$ matrices (usually equipped with the operator norm) and by $S^1_n$ the same space equipped with the trace-class norm. For a functional $\phi$ on an algebra $A$ and an element $a\in A$ we write $a \cdot\phi$ for the functional given by the formula $(a \cdot\phi)(b) = \phi(ba)$, $b \in A$. The tensor flip will be denoted by $\Sigma$.  All scalar products are linear on the right.

\subsection{Compact quantum groups}\label{subsec-cqg}
Throughout the paper $\QG$ will be a compact quantum group in the sense of Woronowicz (\cite{wor1}, \cite{worLH} -- for the facts stated below we refer also to \cite{BMT}). It is studied via its `algebras of functions': a Hopf $^*$-algebra $\Pol(\QG)$ spanned by the coefficients of finite-dimensional unitary representations of $\QG$, $\C(\QG)$, the reduced version of the algebra of continuous functions on $\QG$, and its universal counterpart, $\C_u(\QG)$. The latter two are unital \cst-algebras; they can be constructed from $\Pol(\QG)$ via the natural completion procedures. The reduced one arises from completing $\Pol(\QG)$ to a \cst-algebra via the GNS representation associated to the \emph{Haar state} $h$, unique left-and right invariant state on $\Pol(\QG)$:
\[ (h \ot \id) \circ  \Com = (\id \ot h) \circ \Com = h(\cdot) 1.\]
The coproduct $\Com$ of $\Pol(\QG)$ extends continuously to a coassociative unital $^*$-homomorphism from $\C(\QG)$ to $\C(\QG) \ot \C(\QG)$, denoted by the same symbol.
The state $h$ is faithful on $\Pol(\QG)$, as is its extension to $\C(\QG)$. The GNS Hilbert space of $h$ is usually denoted simply  $\ltwo (\QG)$, so that  we have $\C(\QG) \subset B(\ltwo(\QG))$. We will need later as well the von Neumann algebra $\Linf(\QG):=\C(\QG)''$, playing the role of the algebra of bounded measurable functions on $\QG$; it is also naturally represented on $\ltwo(\QG)$. On the other hand $\C_u(\QG)$ is the universal $\cst$-completion of $\Pol(\QG)$, so that in particular we also have the coproduct on the universal level and a natural quotient map $\Lambda: \C_u(\QG) \to \C(\QG)$. If the latter is injective (equivalently, the Haar state remains faithful on $\C_u(\QG)$), we say that $\QG$ is \emph{coamenable}. Let us also recall that the coproduct of $\Pol(\QG)$ extends also in a natural way to (injective) maps on $\C_u(\QG)$ and on $\Linf(\QG)$. Thus the spaces $\C_u(\QG)^*$, $\C(\QG)^*$ and $\Linf(\QG)^*$  and $\Pol(\QG)'$ (where the last one denotes the vector space dual) become algebras with respect to the \emph{convolution product}
\[ \omega \star \phi:= (\omega \ot \phi) \circ \Com.\]
Each state on $\PolQG$ (i.e.\ a unital functional which takes non-negative values on elements of the form $x^*x$, $x \in \PolQG$) extends uniquely to a state on $\C_u(\QG)$ -- see for example \cite[Theorem 3.3]{BMT}. Note also that those Hopf $^*$-algebras which arise as $\PolQG$ for a compact quantum group $\QG$ admit an intrinsic characterization as so-called \emph{CQG algebras} (see \cite{DiK}).

Finite-dimensional unitary representations of $\QG$ are unitary matrices $U:=[u_{ij}]_{i,j=1}^n \in M_n (\Pol(\QG))\approx M_n \odot \Pol(\QG)$ such that $n \in \mathbb{N}$ and
\[ \Com (u_{ij})= \sum_{k=1}^n u_{ik} \ot u_{kj}, \;\;\;i, j=1,\ldots,n \]
(from now on we will simply call them  representations of $\QG$). A representation $U\in M_n (\Pol(\QG))$ is said to be irreducible if the only scalar matrices $T \in M_n$ such that $U (T \ot 1) = (T \ot 1) U$ are scalar multiples of the identity; two representations $U, V \in M_n (\Pol(\QG))$ are called equivalent if there is a unitary matrix $T \in M_n$ such that $U (T \ot 1) = (T \ot 1) V$. The set of equivalence classes of irreducible representations of $\QG$ will be denoted by $\Irr(\QG)$; we will always assume that for a given $\alpha \in \Irr(\QG)$ we fix a representative $U^{\alpha}= [u^{\alpha}_{ij}]_{i,j=1}^n \in M_{n_{\alpha}} (\Pol(\QG))$. The Peter-Weyl theorem for compact quantum groups says (in particular) that the set $\{u^{\alpha}_{ij}:\alpha \in \Irr(\QG), i,j=1,\ldots, n_{\alpha}\}$ is a linear basis of $\Pol(\QG)$. Moreover for each $\alpha \in \Irr(\QG)$ there exists a unique positive matrix $Q_{\alpha} \in GL(n_{\alpha})$ such that $\Tr(Q_{\alpha}) = \Tr(Q_{\alpha}^{-1}):=d_{\alpha}\geq n_{\alpha}$ and we have for all $\alpha, \beta \in \Irr(\QG)$
\[h\left(u_{ij}^{\alpha} (u_{kl}^{\beta})^* \right) = \delta_{\alpha \beta} \delta_{ik} \frac {(Q_{\alpha})_{l,j}}{d_{\alpha}};\]
\[h\left((u_{ij}^{\alpha})^* u_{kl}^{\beta} \right) = \delta_{\alpha \beta} \delta_{jl} \frac {(Q_{\alpha}^{-1})_{k,i}}{d_{\alpha}}.\]
The quantum group $\QG$ is said to be \emph{of Kac type} if $Q_\alpha=I_{M_{n_{\alpha}}}$ for each $\alpha \in \Irr(\QG)$ or, equivalently, the Haar state $h$ is a trace. The number $d_{\alpha}$ is called the \emph{quantum dimension} of the representation $U^{\alpha}$. One can always assume, by choosing suitable representatives $U^{\alpha}$, that all the matrices $Q_{\alpha}$ are diagonal, in which case we denote their diagonal entries simply by $q_i(\alpha)$, $i=1, \ldots, n_{\alpha}$. We will do so later without further comment.

The \emph{antipode} $S$ on $\PolQG$ is determined by the formula
\[ S(u^{\alpha}_{ij})= (u_{ji}^{\alpha})^*, \;\;\;\;\alpha \in \Irr(\QG), i,j=1,\ldots, n_{\alpha}.\]
It need not extend continuously to $\C(\QG)$ or $\C_u(\QG)$, but the \emph{unitary antipode} $R$, given by the formula
\[ R(u_{ij}^{\alpha}) = \left( \frac{q_j(\alpha)}{q_i(\alpha)} \right)^{\frac{1}{2}} (u_{ji}^{\alpha})^* \;\;\;\alpha \in \Irr(\QG), i,j=1,\ldots, n_{\alpha},\]
does extend to an isometric map. The matrix $\big(R(u^\alpha_{kj})\big)_{1\le j,k\le n_\alpha}\in M_n(\Pol(\QG))$ is again an irreducible representation of $\QG$ called the contragredient representation of $U^\alpha$, which we denote by $\bar{\alpha} \in \Irr(\QG)$. We will choose the representatives such that
	$$u^{\bar{\alpha}}_{jk}=R(u^\alpha_{kj}),\;\; \alpha\in\Irr(\QG),\; 1\le j,k\le n_\alpha.$$
Note that $\QG$ is of Kac type if and only if $S=R$.
We will later use also two natural automorphism groups of $\PolQG$ (viewed as a unital $^*$-algebra): the \emph{modular automorphism group} $(\sigma_t)_{t \in \br}$ and the \emph{scaling automorphism group} $(\tau_t)_{t \in \br}$: they are given respectively by the formulas ($t \in \br, \alpha \in \Irr(\QG), k,l=1,\ldots, n_{\alpha}$):
\[ \sigma_t(u_{kl}^{\alpha}) = (q_k(\alpha)q_l(\alpha)) ^{it} u_{kl}^{\alpha}.\]
\[ \tau_t(u_{kl}^{\alpha}) = \left( \frac{q_k(\alpha)}{q_l(\alpha)} \right)^{it} u_{kl}^{\alpha}. \]
(note both of these have natural extensions for real $t$ replaced by a complex number).

Finally, we recall the counit $\Cou :\PolQG \to \mathbb{C}$ given by $\Cou(u^\alpha_{ij}) = \delta_{ij}$, $\alpha \in \Irr(\QG)$, $1\le i,j\le n_\alpha$. Note that we have
	$$(id\otimes \Cou) \circ \Delta = (\Cou\otimes id) \circ \Delta = id.$$

\vspace*{0,5 cm}
To each compact quantum group $\QG$ one can associate three more formally different compact quantum groups $\QG'$, $\QG^\cop$ and $\QG^\op$. In its von Neumann algebraic version, the \emph{commutant} quantum group $\QG'$ is defined as the commutant $\Linf(\QG')=\Linf(\QG)'$ in $B(\ltwo(\QG))$ with the coproduct $\Delta_{\QG'}(x)= (J\ot J)\Delta(JxJ)(J\ot J)$ for $x\in \Linf(\QG')$, where $J$ is the standard involution on $\ltwo(\QG)$. The \emph{opposite} compact quantum group $\QG^\op$  is obtained from $\QG$ by replacing the multiplication $m$ on $\Pol(\QG)$ (or $\C(\QG)$, $C_u(\QG)$, $\Linf(\QG)$) with the opposite multiplication $m_\op=m\circ\Sigma$. Similarly, the \emph{co-opposite} compact quantum group $\QG^\cop$ is obtained from $\QG$ by replacing the coproduct $\Delta$ on $\Pol(\QG)$ (or $\C(\QG)$, $C_u(\QG)$, $\Linf(\QG)$) with the opposite coproduct $\Delta_\cop=\Sigma\circ\Delta$. This does not change the unit, the co-unit, or the Haar state, but the antipode and the Q-matrices have to be replaced by their inverses.

 In fact one always has $\QG^\cop \approx \QG^\op$ and if $\QG$ is of Kac type, then $\QG' \approx \QG^\op$. The first statement follows from the fact that each compact quantum group admits a unitary antipode $R$, the second is just the fact that if $\tau$ is a faithful normal trace on a von Neumann algebra $\mlg$, $(\pi_{\tau}, \Hil_{\tau}, \Omega_{\tau})$ is the GNS triple for $(M, \tau)$, and $J$ is the standard involution on $\Hil_{\tau}$ (so that $J(\pi_{\tau}(x)\Omega)=\pi_{\tau}(x^*)\Omega$ for all $x \in M$, then the map $\gamma:\pi_{\tau}(M) \to \pi_{\tau}(M)'$ given by the formula  $\gamma(z) = J z^* J, \; z \in \pi_{\tau}(M)$, is an anti-isomorphism. Note finally that our $\QG^\cop$ is the same as $\QG^\op$ in \cite{KShom}.

\subsection{Discrete quantum groups and the Fourier algebra of a compact quantum group} \label{discreteintro}

Given a compact quantum group $\QG$ its dual, discrete quantum group $\hQG$  is defined via its algebras of functions, $c_0(\hQG)$ and $\ell^{\infty}(\hQG)$, arising in the following way:
\[ c_0(\hQG):= c_0\,\text{-}\bigoplus_{\alpha \in \Irr(\QG)} M_{n_{\alpha}}, \;\;\; \ell^{\infty}(\hQG) :=\ell^\infty \text{-} \bigoplus_{\alpha \in \Irr(\QG)} M_{n_{\alpha}}.\]
Note that if $\QG=G$ happens to be a classical compact group, then $\ell^{\infty}(\hQG)=\textup{VN}(G)$, the group von Neumann algebra of $G$.

The above algebras $c_0(\hQG)$ and $\ell^{\infty}(\hQG)$ are also equipped with natural coproducts, denoted by $\hh{\Com}$, which are in a sense dual to the products of $\QG$. One way to introduce $\hh{\Com}$ is via the \emph{multiplicative unitary} $U$ in $\Linf(\QG) \wot \ell^{\infty}(\hQG)$, given by the formula
	$$U = \bigoplus_{\alpha \in \text{Irr}(\Gb)}U^{\alpha}$$
which satisfies the equality
\begin{equation}  \label{comU} (\hh{\Delta}\otimes I)U = U_{12}U_{13}, \end{equation}
determining $\hh{\Com}$.
We will give a more explicit description of $\hh{\Com}$ in Proposition \ref{prop1}. We note that $\hh{\Com}$ can be extended in a natural way to
	\begin{equation}\label{eq-unbdd-extension-coprod}
	\hh{\Com}: \prod_{\alpha \in \text{Irr}(\QG)}M_{n_\alpha} \to \prod_{\beta, \gamma \in \text{Irr}(\QG)}M_{n_\beta} \otimes M_{n_\gamma},
	\end{equation}
where $\prod_{\alpha \in \text{Irr}(\QG)}M_{n_\alpha}$ denotes the space of all sequences of matrices $(X_\alpha)_{\alpha \in \text{Irr}(\QG)}$ without any norm condition on $X_\alpha$.  The set of those elements of $\ell^{\infty}(\hQG)$ which are finitely supported will be denoted $c_{00}(\hQG)$. The space $c_{00}(\hQG)$ is equipped with $\widehat{h}$, the \emph{left Haar weight} on $\widehat{\Gb}$, given by the formula
	$$\widehat{h}((X_\alpha)_{\alpha \in \Irr(\QG)}) = \sum_{\alpha \in \text{Irr}(\QG)}d_\alpha\text{Tr}(Q_\alpha X_\alpha)$$
for all  $(X_\alpha)_{\alpha \in \Irr(\QG)}$ in $c_{00}(\hQG)$. To any element $X = (X_\alpha)_{\alpha \in \Irr(\QG)} \in c_{00}(\hQG)$ we can associate an element $X\cdot \widehat{h}$, which is a bounded linear functional on  $\ell^{\infty}(\hQG)$ given by the natural formula $(X\cdot \widehat{h})(a) = \hh{h}(aX), \;\;\; a \in \ell^{\infty}(\hQG).$

The space $c_{00}(\hQG)$ is also equipped with the {\it antipode} (of the dual quantum group $\hQG$) $\widehat{S} : c_{00}(\hQG) \mapsto c_{00}(\hQG)$ characterized by the following property.
	\begin{equation}\label{eq-antipode-dual}
	(I \otimes \widehat{S})(U^*) = U.
	\end{equation}	
Once we fix a canonical choice of basis we can assume that $\widehat{S}$ is concretely described as follows.
	\begin{equation}\label{eq-antipode-dual-transpose}
	\widehat{S}(X)_\alpha = X^t_{\bar{\alpha}},\;\; \alpha \in \Irr(\QG)
	\end{equation}	
for any $X = (X_\alpha)_{\alpha \in \Irr(\QG)} \in c_{00}(\hQG)$, where $\bar{\alpha}$ refers to the conjugate representation of $\alpha \in \Irr(\QG)$. See \cite[3.16]{PodWo}. Note that $X^j$ in \cite{PodWo} is the same as $\overline{X}$ for any scalar matrix $X$ by identifying $H_{\bar{\alpha}}$ with $\overline{H_\alpha}$.

In this paper we focus more on the predual of $\ell^{\infty}(\hQG)$, which we denote by $\ell^1(\hQG)$, which becomes an algebra with a convolution type product, namely $\widehat{\Delta}_*$. By analogy with the classical harmonic analysis we denote the algebra $\ell^1(\hQG)$ also as $\A(\QG)$ and call it the \emph{Fourier algebra of $\QG$}. Moreover, the norm structure of $\A(\QG)$, using the left Haar weight as the basis of the duality, can be described as follows.
	\begin{equation} \label{Fourier norm} ||X\cdot \widehat{h} ||_{A(\Gb)} = \sum_{\alpha \in \Irr (\QG)} d_\alpha \|X_\alpha Q_\alpha\|_{S^1_{n_\alpha}},\end{equation}
where $X = (X_\alpha)_{\alpha \in \Irr(\QG)} \in c_{00}(\hQG)$. Thus the algebra $\A(\QG)$ can be identified with the $\ell^1$-type direct sum of the spaces $S^1_{n_{\alpha}}$ (this gives it also an operator space structure -- see Subsection \ref{subsec-A_Delta}).

The term ``Fourier algebra'' has a close connection with Fourier transform. Recall that the Fourier transform on $\Pol(\Gb)$ is given by
	$$\mathcal{F} : \PolQG \to \ell^{\infty}(\hQG),\;\; x\mapsto (x\cdot h \otimes I)(U^*),$$
where $h$ is the Haar state on $\Gb$. Moreover, the extension of the classical Plancherel theorem takes the form of the equality
	$$\la y, x \ra_{L^2(\Gb)} (= h(y^*x)) = \la \F(y), \F(x)\ra_{L^2(\widehat{\Gb})} (=\widehat{h}(\F(y)^*\F(x))),\;\;\; x, y \in \PolQG.$$
Working as usual with the diagonal matrices $Q_\alpha$ we  can easily check that
	$$\F(u^{\alpha}_{ij}) = \frac{1}{d_\alpha q_i(\alpha)} e^{\alpha}_{ji},$$
where $e^{\alpha}_{ji}$ is a suitable matrix unit in $M_{n_{\alpha}}$. For a detailed description of the properties of $\F$ we refer, for example, to \cite{Simeng}.  The Fourier transform allows us to embed $\Pol(\QG)$ into $\A(\QG)$  as follows:
	\begin{equation}\label{eq-embedding}
	\PolQG \hookrightarrow\A(\QG),\;\; x \mapsto \F(x)\cdot \widehat{h}.
	\end{equation}
This embedding gives us the duality of $(\PolQG, \prod_{\alpha \in \text{Irr}(\QG)}M_{n_\alpha})$% where as before $\prod_{\alpha \in \text{Irr}(\QG)}M_{n_\alpha}$ denotes the space of all sequences of matrices $(X_\alpha)_{\alpha \in \text{Irr}(\QG)}$ without any norm condition on $X_\alpha$
 as follows.
	\begin{equation}\label{eq-duality}
		\la u^{\alpha}_{ij}, e^{\beta}_{kl} \ra
		:= (\F(u^{\alpha}_{ij})\cdot \widehat{h})(e^{\beta}_{kl} )
		= \frac{1}{d_\alpha q_i(\alpha)}(e^{\alpha}_{ji}\cdot \widehat{h})(e^{\beta}_{kl})
		= \frac{1}{d_\alpha q_i(\alpha)}\widehat{h}(e^{\beta}_{kl}e^{\alpha}_{ji})
		= \delta_{ik}\delta_{jl}\delta_{\alpha \beta}.
	\end{equation}
In other words, if we identify $\PolQG$ and $c_{00}(\hQG)$ via the Fourier transform, then the embedding $\PolQG \hookrightarrow\A(\QG),\;\; x \mapsto \F(x)\cdot \widehat{h}$ transfers the duality $(c_{00}(\hQG), \prod_{\alpha \in \text{Irr}(\QG)}M_{n_\alpha})$ with respect to the the Haar weight $\widehat{h}$ to the above duality.

%({\bf Do we need this??}) Note another useful formula expressing the identification between $\A(\QG)=\overline{\Pol(\QG)}$ and $\ell^1(\widehat{\QG})=\ell^\infty(\widehat{\QG})_*\subseteq \ell^\infty(\widehat{\QG})^*$:
%\[
%\la u^\alpha_{jk}, \,\cdot\, \ra \equiv \frac{e^\alpha_{kj}}{d_\alpha q_j(\alpha)} \cdot \widehat{h}.
%\]

\subsection{Closed quantum subgroups of compact quantum groups}\label{subsection-closedQsubgroups}

A \emph{morphism from a compact quantum group $\QH$ to a compact quantum group $\QG$} is represented by a unital $^*$-homomorphism $\pi:\Pol(\QG) \to \Pol(\QH)$ intertwining the respective coproducts:
\[ (\pi \ot \pi) \circ \Com_{\QG} =   \Com_{\QH} \circ \pi.\]
If $\pi$ is surjective, then we say it identifies $\QH$ with a \emph{(closed) quantum subgroup} of $\QG$. Such maps $\pi$ extend to unital $^*$-homomorphisms from $\C_u(\QG)$ to $\C_u(\QH)$, but not neccessarily to the reduced \cst-algebra $\C(\QG)$. Note that the fact that $\QG$ and $\QH$ are isomorphic means simply that $\PolQG$ and $\Pol(\QH)$ are isomorphic as Hopf $^*$-algebras.

In the above situation we have another canonical $*$-homomorphism $\iota: \ell^{\infty}(\widehat{\QH}) \to \ell^{\infty}(\hQG)$ satisfying
	$$\hat{\Delta}_{\Gb} \circ (\iota \otimes \iota) = \iota \circ \hat{\Delta}_{\Hb},$$
where $\hat{\Delta}_{\Gb}$ and $\hat{\Delta}_{\Hb}$ are coproducts of $\widehat{\QG}$ and $\widehat{\QH}$, respectively. The above two maps $\pi$ and $\iota$ are closely related by the following identity.
	\begin{equation}\label{eq-two-homomorphisms}
	(\pi\otimes I)\mathbb{U}_{\Gb} = (I\otimes \iota)\mathbb{U}_{\Hb},
	\end{equation}
where $\mathbb{U}_{\Gb}\in \textup{M}(\C^u(\QG) \ot c_0(\hQG)) $ and $\mathbb{U}_{\Hb}\in \textup{M}(\C^u(\Hb) \ot c_0(\hQH))$ are the `semi-universal' multiplicative unitaries, respectively of $\QG$ and of $\QH$, and we use the fact that $c_b(\hQH)$ is weak$^*$-dense in $\ell^{\infty}(\hQH)$. We refer to \cite{DawKasSkaSol} for details.

\subsection{Operator spaces, a Fourier-like algebra $A_\Delta(G)$ of a compact group $G$ and its Banach algebra properties in the operator space category}\label{subsec-A_Delta}
In the second part of the paper we will use the standard notions of operator space theory, as can be found for example in the book \cite{Pisier}.

Recall that an operator space is a norm closed subspace $E$ of $B(H)$ for some Hilbert space $H$. Then, we have a  natural norm structure on $M_n(E) = M_n \odot E$ regarded as a subspace of $B(H^n)$ for each $n \in \bn$. A linear map $T: E \to F$ between operator spaces is called {\it completely bounded} if the cb-norm $||T||_{cb} := \sup_{n\ge 1} ||I_n\otimes T: M_n(E) \to M_n(F)||$ is finite. We denote the space of all completely bounded maps from $E$ into $F$ by $CB(E,F)$. We say that $T\in CB(E,F)$ is a {\it complete contraction} if $||T||_{cb}\le 1$ and $T$ is a {\it complete isometry} if $T$ and $T^{-1}$ are complete contractions.

The category of operator spaces is usually called a {\it quantized} version of the category of Banach spaces, and most of the concepts for Banach spaces, such as tensor products, direct sums, and duality, have their counterparts for operator spaces. We will mainly use two tensor products, namely {\it projective} and {\it injective} (or {\it minimal}) tensor product of operator spaces, which we denote by $\widehat{\otimes}$ and $\otimes$, respectively. There are aso several versions of direct sums in the category of operator space. We denote the $c_0$, $\ell^\infty$ and $\ell^1$-direct sums by $c_0\text{-}\bigoplus_{i\in I}E_i$, $\ell^\infty \text{-}\bigoplus_{i\in I}E_i$ and $\ell^1\text{-}\bigoplus_{i\in I}E_i$, respectively, for a family of operator spaces $(E_i)_{i\in I}$. A \cst-algebra $A$ is, of course, an operator space, and its Banach space dual $A^*$ also carries a canonical operator space structure. Moreover,  the predual $\M_*$ of a von Neumann algebra $\M$ can be understood as an operator space via the embedding $\M_* \hookrightarrow \M^*$.

There are some more concrete operator spaces we will encounter frequently. For a Hilbert space $H$ the operator spaces $B(H,\Comp)$ and $B(\Comp, H)$ (as subspaces of $B(H)$) are called the {\it row} and {\it column Hilbert space} on $H$ denoted by $H_r$ and $H_c$, respectively. They are known to be operator space duals of each other, i.e.
	$$H^*_r \approx H_c,\;\; H^*_c \approx H_r$$
completely isometrically. When dim$H=n<\infty$, we simply denote them by $R_n$ and $C_n$. The \cst-algebra $M_n=B(\bc^n)$ and its dual $S^1_n$ are known to be decomposed by the row and column Hilbert spaces, namely
	$$M_n \approx R_n \ot C_n,\;\;  S^1_n \approx C_n \prt R_n$$
completely isometrically via the identification $e_{ij} \mapsto e_i \otimes e_j$. The space of Hilbert-Schmidt $n$ by $n$ matrices, $S_n^2$, is also a Hilbert space. Thus the space $S^2_n$ can be equipped with the row or column operator space structures, which decompose completely isometrically
	 $$(S^2_n)_r \approx R_n \ot R_n,\;\;  (S^2_n)_c \approx C_n \ot C_n$$
via the same identification $e_{ij} \mapsto e_i \otimes e_j$. Note further that we have
	$$CB(R_n, C_n) \approx CB(C_n, R_n) \approx S^2_n$$
isometrically.

We will now briefly describe how we can use the diagonal subgroups in the classical setting (i.e.\ when $G$ is a compact group) to define a Fourier-like algebra, to be denoted by $\A_\Delta(G)$. Recall that the usual  \emph{Fourier algebra of $G$}, as defined by Pierre Eymard in \cite{Eymard} will be denoted by $\A(G)$ and equipped with a natural operator space structure (see \cite{Ruan}).

Let $G$ be a classical compact group. Then the diagonal subgroup $D = \{(x,x) \in G\times G: x\in G\}$ is a closed subgroup of $G\times G$, which can be identified with $G$ via the map $(x,x) \mapsto x\in G$. Thus, averaging over $D$ allows us to consider a closed subalgebra of $\A(G\times G)$ as follows. The averaging map $P$ (over $D$) is given by
	$$P : \A(G\times G) \to \A(G\times G),\;\; f \mapsto P(f)$$
with
	$$P(f)(x,y) = \int_G f(xr,yr)dr,$$
where $dr$ is the Haar measure on $G$. Note that a priori $P$ has only values in the continuous functions on $G \times G$, but it follows from general properties of Fourier algebras that in fact $P( \A(G\times G)) \subseteq \A(G\times G)$.  It is clear that $P$ is a completely contractive idempotent whose range space is
	$$\text{ran}P = \{F\in \A(G\times G): r\cdot F = F,\; \forall r\in G\},$$
where $r\cdot F(x,y) = F(xr,yr)$, $x,y\in G$. It is also clear that the map $P$ satisfies the property
	$$P(F\cdot f) = F\cdot P(f),\;\; f\in A(G\times G), F\in \text{ran}P.$$
The above property tells us that $\text{ran}P$ is a closed subalgebra of $\A(G\times G)$ and $P$ is actually a complete quotient map. Note that the averaging map $P$ is nothing but a restriction of a conditional expectation from $C(G\times G)$ onto the corresponding subalgebra $\{F\in C(G\times G): r\cdot F = F,\; \forall r\in G\}$, which can be easily identified with $C(G\times G/D)$ via the map $F \mapsto M(F)$ with $M(F)((x,y)D) = F(x,y)$, $x,y\in G$.

Now we would like to find a completely isometric copy of the algebra $\text{ran}P \subseteq \A(G\times G)$ as a subalgebra of $\A(G)$, possibly not closed under the norm of $\A(G)$. In order to get such a model we recall the subalgebra $\{F\in C(G\times G): r\cdot F = F,\; \forall r\in G\} \cong C(G\times G/D)$ and the identification $G\times G/D \cong G,\; (x,e)D \mapsto x\in G$ to get $C(G\times G/D)\cong C(G)$. Consequently, combining all the above maps we get
	\begin{equation}\label{eq-Gammamap}
	\Gamma : \A(G\times G) \to C(G\times G/D) \cong C(G),\;\; f\mapsto \Gamma(f) = M\circ P(f).
	\end{equation}
In particular, for $f = u\otimes v \in \A(G\times G)$ we have
	$$\Gamma(u\otimes v)(x) = \Gamma(u\otimes v)((x,e)D) = \int_Gu(xr)v(r)dr = u*\check{v}(x),$$
where $\check{v}(x) = v(x^{-1})$.
Now we define a subalgebra
	\begin{equation}\label{eq-A_Delta-classic}
	\A_\Delta(G) := \text{ran}(\Gamma)
	\end{equation}
of $C(G)$, equipped with the norm
	$$||\Gamma(F)||_{\A_\Delta(G)} := ||P(F)||_{\A(G\times G)}.$$
We further endow $\A_\Delta(G)$ with the operator space structure which makes the map $\Gamma(F) \mapsto P(F)$ a complete isometry. This makes $\Gamma: \A(G\times G) \to \A_\Delta(G)$ a complete quotient map, i.e. the adjoint map $\Gamma^*$ is a complete isometry. In \cite{FSS} it was shown that $\A_\Delta(G)$ is a subalgebra of $C(G)$ with the following norm structure:
	\begin{equation}\label{eq-norm-A_Delta}
	||f||_{\A_\Delta(G)} = \sum_{\pi \in\widehat{G}} d^{\frac{3}{2}}_\pi ||\widehat{f}(\pi)||_2.
	\end{equation}
This explains that $\A_\Delta(G)$ is even a (non-closed) subalgebra of $\A(G)$.

We close this subsection with some known Banach algebraic properties of $A_\Delta(G)$ in operator space category.  A Banach algebra $\mathcal{A}$ with an operator space structure is called a {\it completely contractive Banach algebra} if its algebra multiplication $m: \mathcal{A} \otimes_\gamma \mathcal{A} \to \mathcal{A}$ extends to a completely contractive map $m: \mathcal{A} \prt \mathcal{A} \to \mathcal{A}$. The algebras $A(G)$ and $A_\Delta(G)$ are known to be completely contractive Banach algebras. Recall that a (completely contractive) Banach algebra $\mathcal{A}$ is called (operator) amenable if every (completely) bounded derivation $D: \mathcal{A} \to X^*$ is inner, for any Banach (operator) $\mathcal{A}$-bimodule $X$. We say that $\mathcal{A}$ is (operator) weakly amenable if the same statement holds for the specific choice of $X=\mathcal{A}$. The algebra $\A_\Delta(G)$ reflects properties of $G$ in a different way from the way $\A(G)$ does, namely through amenability-type properties  in the operator space category. For example, it has been shown that (\cite[Theorem 4.1, Theorem 4.2]{FSS2}) $\A_\Delta(G)$ is operator amenable if and only if $A(G)$ is amenable (in the Banach space category) if and only if $G$ is virtually abelian. Moreover, $\A_\Delta(G)$ is operator weakly amenable if and only if $\A(G)$ is weakly amenable if and only if $G_e$, the connected component of the identity of $G$, is abelian.

\section{Idempotent states on compact quantum groups -- non-coamenable case}

In this short section we show how the correspondence between the idempotent states on compact quantum groups and expected coidalgebras, established in \cite{FS} in the coamenable case, extends to general compact quantum groups.

%Recall that $\C_u(\QG)$ is the universal completion of the unital $^*$-algebra $\PolQG$, $\C(\QG)$ is the GNS $\cst$-completion of $\PolQG$ with respect to the (faithful) Haar state, and $L^{\infty}(\QG)$ is the von Neumann algebraic GNS completion of $\PolQG$ with respect to the  Haar state. The coproduct on each of these algebras will be denoted by $\Com$.

We say that $A\subset \Pol(\QG)$ (respectively, $\alg \subset\C(\QG)$, $\mathfrak{A}\subset L^{\infty}(\QG)$) is a \emph{right coidalgebra} if it is a unital $^*$-algebra (respectively, norm closed and weak$^*$-closed) such that $\Com(A)\subset  \PolQG \odot A$ (respectively,  $\Com(\alg)\subset \C(\QG) \ot \alg  $, $\Com(\mathfrak{A})\subset  L^{\infty}(\QG) \wot \mathfrak{A} $). We say that a right coidalgebra $A\subset \Pol(\QG)$ (respectively, $\alg \subset\C(\QG)$, $\mathfrak{A}\subset L^{\infty}(\QG)$) is \emph{expected} if there exists a Haar state preserving conditional expectation (i.e.\ a completely positive unital idempotent map) from $\PolQG$ onto $A$ (respectively from $\C(\QG)$ onto $\alg$, from $L^{\infty}(\QG)$ onto $\mathfrak{A}$).

\begin{lem}\label{Polinv}
Suppose that $x\in L^{\infty}(\QG)$ and $\Com(x) \subset \PolQG \odot L^{\infty}(\QG)$. Then $x \in \PolQG$.
\end{lem}
\begin{proof}
This is likely well known, and essentially follows as in Proposition 2.2 of \cite{Soltan}. We reproduce the proof for completeness. For each $\alpha$ in $\Irr(\QG)$ consider the normal projection $E_{\alpha}$ from $L^{\infty}(\QG)$ onto $X_{\alpha}:=\tu{span}\{u_{ij}^{\alpha}:i,j=1, \ldots n_{\alpha}\}$, given by the formula
\[E_{\alpha} = \sum_{r=1}^{n_{\alpha}} ( \phi^{\alpha}_{r,r}\ot \id)\circ \Com,\]
where for every $r=1,\ldots, n_\alpha$ the elements $\phi^{\alpha}_{r,r}\in L^{\infty}(\QG)_*$ are determined by equalities
\[ \phi^{\alpha}_{r,r}(u_{i,j}^{\beta}) = \delta_{\alpha, \beta} \delta_{i,r} \delta_{j,r}, \;\;\;\beta\in\Irr(\QG), i,j=1,\ldots, n_{\beta}.\]
In other words, $\phi^{\alpha}_{r,r} \approx e^\alpha_{rr}$ with respect to the duality \eqref{eq-duality}. The fact that $E_{\alpha}(\PolQG)=X_{\alpha}$ can be checked by explicit calculation, and then  $E_{\alpha}(L^{\infty}(\QG))=X_{\alpha}$ follows by normality and density.
Let $x$ be as in the statement of the lemma. Then there exists a finite set $F \subset \Irr(\QG)$ and elements $\{b_{ij}^{\alpha} \in L^{\infty}(\QG):\alpha \in F, i,j=1,\ldots, N_{\alpha}\}$ such that
\[\Com(x)= \sum_{\alpha \in F} \sum_{i,j=1}^{n_{\alpha}} u_{ij}^{\alpha} \ot b_{ij}^{\alpha}.\]
Computing $(\id \ot \Com) (\Com(x))$ and $(\Com \ot \id)(\Com(x))$, followed by using then linear independence of $u_{ij}^{\alpha}$ yields the equality
\[ \Com(b_{ij}^{\alpha}) =\sum_{k=1}^{n_{\alpha}} u_{jk}^{\alpha} \ot b_{ik}^{\alpha}, \;\;\; \alpha \in F, i,j=1,\ldots, n_{\alpha}.\]
Then, we have
	$$E_\alpha(b^\alpha_{ij}) = \sum_{r=1}^{n_{\alpha}} ( \phi^{\alpha}_{r,r}\ot \id)\circ \Com(b^\alpha_{ij}) = \sum_{r, k=1}^{n_{\alpha}} \phi^{\alpha}_{r,r}(u^\alpha_{jk}) b^\alpha_{ik} = b^\alpha_{ij},$$
so that we have $b_{ij}^{\alpha}\in X_{\alpha} \subset \PolQG$. Further consider $x'=\sum_{\alpha \in F}\sum_{i=1}^{n_{\alpha}} b_{ii}^{\alpha}$. Then the equality displayed above shows that $\Com(x)=\Com(x')$, so, as $\Com$ is injective, $x=x' \in \PolQG$.
\end{proof}

We will also need the following lemma, which is a consequence of the main result of \cite{Matt} (and is also stated explicitly as Lemma 3.4 in \cite{Bra}).

\begin{lem}\label{normalext}
Let $\varphi: \PolQG\to \bc$ be a state. Then the convolution operator $L_{\varphi}:\PolQG \to \PolQG$ given by the formula $(\id \ot \varphi)\circ \Com$ extends to a completely positive normal map $\tilde{L}_{\varphi}:L^{\infty}(\QG) \to L^{\infty}(\QG)$ (so also to a completely positive unital map on $\C(\QG)$).
\end{lem}

The following theorem is a combination of techniques of \cite{FS} and the above results.

\begin{tw}
Let $\QG$ be a compact quantum group.
There is a one-to-one correspondence between the following objects:
\begin{rlist}
\item idempotent states $\varphi$ on $\Pol(\QG)$;
\item idempotent states $\tilde{\varphi}$ on $\C_u(\QG)$;
\item expected right (equivalently, left) coidalgebras $\Alg$ in $\Pol(\QG)$;
\item expected right (equivalently, left) coidalgebras $\alg$ in $\C(\QG)$;
\item expected right (equivalently, left) coidalgebras $\mathfrak{A}$ in $L^{\infty}(\QG)$.
\end{rlist}
It is given by the following relations: $\tilde{\varphi}$ is a continuous extension of $\varphi$; $\Alg = ( \id \ot\varphi )\circ\Com(\Pol(\QG))$, $\alg$ is the norm closure of $\Alg$ in $\C(\QG)$, $\mathfrak{A}$ is the weak $^*$-closure of $\Alg$ in $L^{\infty}(\QG)$ (the corresponding left versions arise as $(\varphi \ot \id )\circ\Com(\Pol(\QG))$, etc.).
\end{tw}
\begin{proof}
We describe the relevant passages for the `right' version of the above equivalences:

(i)$\Longleftrightarrow$(ii)

This follows immediately from the one-to-one correspondence between states on $\Pol(\QG)$ and $\C_u(\QG)$ mentioned in Subsection \ref{subsec-cqg} and the fact that the extension/restriction procedure preserves the convolution products.

(i)$\Rightarrow$(iii)

Consider the map $E_{\varphi}:=(\id \ot\varphi )\circ\Com:\PolQG \to \PolQG$. Lemma 3.1 in \cite{FS} implies that we must have
$E_{\varphi}(E_{\varphi}(x)E_{\varphi}(y))=E_{\varphi}(x)E_{\varphi}(y)$ for all $x,y \in \PolQG$; in particular $\Alg:=E_{\varphi}(\PolQG)$ is a unital $^*$-subalgebra of $\PolQG$ and $E_{\varphi}$ is the $h$-preserving conditional expectation onto $\Alg$. As
\[ \Com \circ E_{\varphi} = (\id \ot E_{\varphi}  ) \circ \Com,\]
$\Alg$ is a right coidalgebra.

 (iii)$\Rightarrow$(i)

 Suppose that $\Alg \subset \PolQG$ is an expected right coidalgebra and $E:\PolQG \to \Alg$ is the corresponding $h$-preserving conditional expectation. Then, arguing as in the proof of Theorem 4.1 of \cite{FS}, or as in the proof of Proposition 3.3 of \cite{PekkaAdam}, we see first that for any linear functional $\nu:\PolQG \to \bc$ we have $(\nu \ot \id)\circ \Com \circ E = E \circ (\nu \ot \id)\circ \Com $ and then that it implies that $\Com \circ E = (\id \ot E  ) \circ \Com$. This implies that $E=(\id \ot\varphi )\circ\Com$, where $\varphi = \Cou \circ E$ is an idempotent state.

(iii)$\Rightarrow$(iv),(v)

From the argument above we see that each expected right coidalgebra $\Alg$ is of the form $(\id \ot\varphi )(\Com (\PolQG))$ and the corresponding $h$-preserving conditional expectation is $E:=(\id \ot\varphi )\ot \Com$. It is clear that the respective norm closure of $\Alg$ in  $\C(\QG)$, $\alg$, and weak $^*$-closure of $\Alg$ in $L^{\infty}(\QG)$, $\mathfrak{A}$, are right coidalgebras.  Lemma \ref{normalext} implies that $E$ extends to a normal unital completely positive map on $L^{\infty}(\QG)$ (preserving $\C(\QG)$). It is easy to see (using at one point Kaplansky theorem and properties of conditional expectations) that the respective extensions of $E$ are $h$-preserving conditional expectations from $L^{\infty}(\QG)$ onto $\mathfrak{A}$ and from $\C(\QG)$ onto $\alg$.

(v)  $\Rightarrow$ (i)

Here the arguments are similar, so we just discuss. We first repeat the proof in  (iii)$\Rightarrow$(i) (again, see Proposition 3.3 of \cite{PekkaAdam}) to show that if say $E: \Linf(\QG) \to \mathfrak{A}$ is the $h$-preserving normal expectation onto an  expected right coidalgebra $\mathfrak{A}$ then for any $\nu \in L^{\infty}(\QG)_*$ we have $(\nu \ot \id)\circ \Com \circ E = E \circ (\nu \ot \id)\circ \Com $ and deduce that  $\Com \circ E = (\id \ot E  ) \circ \Com$ . The last intertwining equality shows that for any $x \in \PolQG$ we have $\Com(E(x)) \in \PolQG \ot L^{\infty}(\QG)$. By Lemma \ref{Polinv} it implies that $E(x) \in \PolQG$. Thus
$E(\PolQG) \subset \PolQG$. Using again the interwining relation above we see that $E|_{\PolQG}=(\id \ot\varphi )\circ\Com$ for an idempotent state $\phi$ on $\PolQG$.

(iv)  $\Rightarrow$ (iii)

Here the argument is almost identical, the only difference is that for $\alg \subset \C(\QG)$ we use functionals $\nu \in \C(\QG)^*$.

\vspace*{0.5cm}
The left versions follow identically, and the fact that the correspondences are 1-1 follow very easily.

\end{proof}

In fact recently Pekka Salmi and the third named author showed in \cite{PekkaAdamII} that the 1-1 correspondence between objects in (ii), (iv) and (v) persists even in the case when $\QG$ is a \emph{locally} compact quantum group, thus extending  the main results of \cite{PekkaAdam} beyond the coamenable and unimodular case. The methods in the locally compact case are quite different.

\section{Averaging over the `diagonal subgroup'}

The aim of this section is the development of the averaging procedure (represented by an idempotent state on $\C_u(\QG) \ot \C_u(\QG))$) corresponding classically to the integration over the diagonal subgroup in $\QG \times \QG$. One can present the construction working either with $\QG^{\op}$ or with $\QG^\cop$ -- we will do the latter, to preserve compatibility with \cite{KShom}. Write $\Alg$ for $\Pol(\QG)$, $\Alg^{\cop}$ for $\Pol(\QG^{\cop})$ and $\Alg^{\op}$ for $\Pol(\QG^{\op})$.

The formal identity map from $\Alg$ to $\Alg^{\cop}$ will be denoted by $\widetilde{}\,$; thus, denoting the tensor flip map on $\Alg \ot \Alg$ by $\Sigma$ we have
\[ \Com^{\cop} \circ \,\widetilde{}\, = (\,\widetilde{}\, \circ \widetilde{}\,)\circ \Sigma \circ \Com.\]
The map $\,\widetilde{}\,:\Alg \to \Alg^{\cop}$ is obviously a vector space isomorphism. Note also that all the algebras of the form say $\Alg \odot \Alg^{\cop}$ are $CQG$-algebras, so in particular we have a natural notion of states (normalised, positive functionals) on them, and also natural convolution products on the corresponding state spaces.  The Haar state of $\QG$ will be denoted as before  by $h$, and the coproduct of $\Alg \odot \Alg^{\cop}$ by $\ComGG$.

Throughout this section, apart from Proposition \ref{Kac=expected}, we assume that $\QG$ is of Kac type. Note that it  means in particular that the antipode $S$ commutes with the adjoint operation.

We will employ regularly the Sweedler notation: $\Com(a) = a_{(1)} \ot a_{(2)}$, $a \in \Alg$, and also  occasionally the `legs' notation for elements of multiple tensor products.

\begin{prop}\label{Diagonal}
The formula
\[\phi_D(a \otimes \widetilde{b}) = h(aSb), \;\;\; a, b \in \Alg,\]
defines (by the linear extension) an idempotent state on $\Alg \odot \Alg^{\cop}$.
\end{prop}
\begin{proof}
Normalisation $\phi_D(1 \ot 1) = 1$ is obvious. Further for $a, b \in \Alg$,
\begin{align*}\phi_D((a \ot \widetilde{b})^* (a \ot \wt{b})) &= \phi_D(a^*a \ot \widetilde{b^*b}) = h(a^*a S(b^*b)) = h(a^* a S(b) S(b)^*) \\&= h(aS(b) S(b)^* a^* ) \geq 0, \end{align*}
which, by linearity, implies that $\phi_D$ is positive on $\Alg \odot \Alg^{\cop}$.
Finally we have
\begin{align*} (\phi_D \star \phi_D)(a \ot \widetilde{b}) &= (\phi_D \ot \phi_D) (\ComGG(a \ot \widetilde{b})) = (\phi_D \ot \phi_D) (a_{(1)} \ot \wt{b_{(2)}} \ot a_{(2)} \ot \wt{b_{(1)}})
\\&= h(a_{(1)} S(b_{(2)})) h ( a_{(2)} S(b_{(1)})) = h(a_{(1)} (Sb)_{(1)}) h ( a_{(2)} (Sb)_{(2)}) \\&= h((aSb)_{(1)}) h((aSb)_{(2)})
= h(h((aSb)_{(1)}) (aSb)_{(2)})= h(aSb)
\\&=  \phi_D(a \ot \widetilde{b})
\end{align*}
\end{proof}

The above fact implies that the map
\[E_{\phi_D}:= (\id_{\Alg \ot \Alg^{\cop}} \ot \phi_D) \circ \ComGG:\Alg \odot \Alg^{\cop} \to \Alg \odot \Alg^{\cop}\]
is a conditional expectation, preserving the Haar state (see \cite{FS} or the previous subsection) -- in particular the image of $E_{\phi_D}$ is a unital $^*$-algebra.

Consider the following (injective!) map
\[ \widetilde{\Com}: \Alg \to \Alg \odot \Alg^{\cop}, \;\;\; \widetilde{\Com} = (\id_{\Alg} \ot \, \wt{}\,) \circ \Com. \]
In \cite{KShom} it is shown that the image of this map (or rather its von Neumann algebraic incarnation) is an `embeddable quantum homogeneous space' of $\QG \times \QG^{\cop}$.  In other words, the coproduct $\ComGG$ has the following property:
\[ \ComGG(\wt{\Com}(\Alg)) \subset \Alg \odot \Alg^{\cop} \odot \wt{\Com}(\Alg).\]
This fact follows also from the next theorem, note however that the observation of \cite{KShom}  does not require the Kac assumption and is valid (when formulated in the von Neumann algebraic language) for general \emph{locally compact} quantum groups.

Recall that the convolution product on $\Alg$ is defined by the formula (see for example \cite{BhSS})
\[ a \star b = h(S(b_{(1)}) a) b_{(2)}, \;\;\; a, b \in \Alg. \]

\begin{tw}\label{equalspaces}
We have the following equality (of subspaces of $\Alg \odot \Alg^{\cop}$):
\[ \widetilde{\Com} ( \Alg) = E_{\phi_D} (\Alg \odot \Alg^{\cop})\]
\end{tw}
\begin{proof}
We first show the inclusion $\subset$. Let $a \in \Alg$ and consider the following chain of equalities:
\begin{align*}
E_{\phi_D} (\widetilde{\Com}(a)) &= (\id_{\Alg\odot \Alg^{\cop}} \ot \phi_D) \circ \ComGG (a_{(1)} \ot \widetilde{a_{(2)}}) =
(\id_{\Alg\odot \Alg^{\cop}} \ot \phi_D)  (a_{(11)} \ot \widetilde{a_{(22)}} \ot a_{(12)} \ot \widetilde{a_{(21)}})
\\&=  a_{(11)} \ot \widetilde{a_{(22)}} h(a_{(12)} Sa_{(21)}) = a_{(1)} \ot \widetilde{a_{(4)}} h(a_{(2)} Sa_{(3)}) =  a_{(1)} \ot \widetilde{a_{(2)}} = \widetilde{\Com}(a) \end{align*}
This proves the first inclusion.

Let then now $a, b \in \Alg$. We claim that the following identity holds (in fact it can also be used to show the first inclusion):
\[ \wt{\Com}(a\star b) = E_{\phi_D} (a \ot \wt{b}).\]
Indeed, we have
\begin{align*} E_{\phi_D} (a \ot \wt{b})&=(\id_{\Alg\odot \Alg^{\cop}} \ot \phi_D) (a_{(1)} \ot \wt{b_{(2)}} \ot a_{(2)} \ot \wt{b_{(1)}} )=
a_{(1)} \ot \wt{b_{(2)}} h( a_{(2)} S(b_{(1)}))  \\&= (\id \ot h) (\Com(a) (1 \ot d)) \ot \wt{b_{(2)}},
\end{align*}
where $d = S(b_{(1)})$ (so that $S(d_{(1)}) = S(S(b_{(1)})_{(1)}) =(S^2(b_{(1)}))_{(2)} = b_{(12)}$ and $d_{(2)} = S(b_{(1)})_{(2)} = S(b_{(11)})$). Continuing the computations, and using the relation (see for example \cite{BhSS})
\[ S((\id \ot h)((1 \ot a) \Com(d)))= (\id \ot h) (\Com(a) (1 \ot d))\]
we get
\begin{align*}E_{\phi_D} (a \ot \wt{b})&=
S((\id \ot h) ((1 \ot a) (\Com(d))) \ot \wt{b_{(2)}}  = S(d_{(1)}) h(a d_{(2)}) \ot \wt{b_{(2)}}
\\&= b_{(12)} h(a S(b_{(11)})) \ot \wt{b_{(2)}} = h(a S(b_{(1)})) b_{(2)}\ot \wt{b_{(3)}}.
\end{align*}
When we compute the left hand side we obtain
\begin{align*}\wt{\Com}(a\star b) &= \wt{\Com} (h(S(b_{(1)}) a) b_{(2)}) = h(S(b_{(1)}) a) b_{(2)} \ot \wt{b_{(3)}}
\\&= h(b_{(1)} Sa) b_{(2)} \ot \wt{b_{(3)}}.
\end{align*}
This ends the proof.
\end{proof}

The above proof shows in particular the following fact.

\begin{cor}\label{gammamap}
The map $\gamma=\widetilde{\Com}^{-1} \circ E_{\phi_D}: \Alg \odot \Alg^{\cop} \to \Alg$ is given by the formula
\[ \gamma(a \ot \wt{b}) = a \star b, \;\;\;\; a, b \in \Alg.\]
Note that the map $\gamma$ is a quantum version of the map $\Gamma$ in \eqref{eq-Gammamap}.
\end{cor}

As observed in \cite{KShom}, the homogeneous space $\widetilde{\Com} ( \Alg)$ is never of quotient type, unless $\Alg$ is abelian. We continue by noting the following easy fact: if $\QH, \QG$ are compact quantum groups, $\QG$ is of Kac type, and $\QH$ is a quantum subgroup of $\QG$, then $\QH$ must be of Kac type. Indeed, this follows for example from the fact that the Kac property is equivalent to the fact that the antipode commutes with the involution, and if $\pi:\Pol(\QG) \to \Pol(\QH)$ is the surjective Hopf $^*$-algebra morphism, then we must have $S_\QH \circ \pi = \pi \circ S_{\QG}$.

\begin{prop}
The following conditions are equivalent:
\begin{rlist}
\item $\widetilde{\Com} (\Alg)$ is of quotient type, i.e.\ there exists a closed quantum subgroup $\QH$ of $\QG \times \QG^{\op}$ with the corresponding morphism $\rho:\Alg \odot \Alg^{\cop} \to \Pol(\QH)$ such that $\widetilde{\Com} (\Alg)=\{x \in \Alg \odot \Alg^{\cop}: (\id\ot \rho) (\ComGG(x)) = x \ot 1\}$;
\item the state $\phi_D$ is a trace;
\item the state $\phi_D$ is a Haar idempotent;
\item $\Alg$ is abelian (i.e.\ $\QG$ is a classical group).
\end{rlist}
\end{prop}
\begin{proof}
(i)$\Longrightarrow$ (ii)
If $\widetilde{\Com}(\Alg)$ is of quotient type (as described in the proposition), then it is expected, and the conditional expectation onto it is of the form $(\id \ot h_{\QH} \circ \rho) \circ \ComGG$. This conditional expectation preserves the Haar state on $\Alg \odot \Alg^{\cop}$, so by the uniqueness and Theorem \ref{equalspaces} we must have $(\id \ot h_{\QH} \circ \rho) \circ \ComGG = E_{\phi_D}$ and further $\phi_D= h_{\QH} \circ \rho$. By the remark before the proposition $\QH$ is of the Kac type, so $h_{\QH}$ is tracial. So is therefore also $\phi_D$.

(ii)$\Longrightarrow$ (iii) This is an easy consequence of Theorem 3.3 in \cite{FST}.

(iii)$\Longrightarrow$ (i) We argue as in the proof of the first implication: if $\phi_D$ is a Haar idempotent, then there exists a a closed quantum subgroup $\QH$ of $\QG \times \QG^{\op}$ with the corresponding morphism $\rho:\Alg \odot \Alg^{\cop} \to \Pol(\QH)$ such that $\phi_D = h_{\QH} \circ \rho$. Then Theorem \ref{equalspaces} ends the proof.

(iv)$\Longrightarrow$ (ii) This is obvious, as then also $\Alg \odot \Alg^{\cop}$ is commutative.

(ii)$\Longrightarrow$ (iv) If $\phi_D$ is a trace, then for any $a, b, c \in \Alg$ we have
\[\phi_D((a \ot \widetilde{Sb})(c \ot \widetilde{1})) = \phi_D ((c \ot \widetilde{1})(a \ot \widetilde{Sb})),\]
that is $\phi_D(ac \ot \widetilde{Sb}) = \phi_D(ca \ot \widetilde{Sb})$, and further $h(acb) = h(cab)$. Faithfulness of $h$ implies then immediately that $ac=ca$.
\end{proof}

\begin{prop}\label{Kac=expected}
Assume that $\QG$ is a general compact quantum group.
The algebra $\widetilde{\Com} ( \Alg)$ is expected (i.e.\ there exists a positive unital projection $E:\Alg \odot \Alg^{\cop}\to \widetilde{\Com} ( \Alg)$ preserving the Haar state of $\Alg \odot \Alg^{\cop}$) if and only if $\QG$ is of Kac type.
\end{prop}
\begin{proof}
Suppose that such a conditional expectation exists. Then it must be of the form $(\phi \ot \id)\circ \ComGG$, where $\phi \in S(\Alg \odot \Alg^{\cop})$ is an idempotent state (this is essentially an algebraic version of Theorem 4.1 of \cite{FS}). As $\phi$ commutes then with the scaling automorphism group of $\Alg \odot \Alg^{\cop}$ (Theorem 3.5 of \cite{FST}), and we have $\Com \circ \sigma_t = (\sigma_t \circ \tau_{-t}) \circ \Com$ for all $t \in \br$, we see that $E$ must commute with the modular automorphism group of $\Alg \odot \Alg^{\cop}$ (again, this is a purely algebraic version of the argument in Corollary 3.6 of \cite{FST}). In other words, $\widetilde{\Com} ( \Alg)$ must be invariant under the action of the modular automorphisms of $\Alg \odot \Alg^{\cop}$. Note however that the modular automorphism group of $\Alg \odot \Alg^{\cop}$ is determined by the formula
\[ \wt{\sigma_t}(a \ot \widetilde{b}) = \sigma_t(a) \ot \widetilde{\sigma_t(b)},\;\;\; a, b \in \Alg, t \in \br.\]
Recall that  we choose for each $\alpha \in \Irr(\QG)$ a representative $(u_{ij}^{\alpha})_{i,j =1}^{n_{\alpha}}$ so that there exist numbers $q_1(\alpha), \ldots, q_{n_{\alpha}}(\alpha) >0$ such that for each $t \in \br$ and $k,l=1,\ldots, n_{\alpha}$ we have
%\[ \tau_t(u_{kl}^{\alpha}) = \left( \frac{q_k(\alpha)}{q_l(\alpha)} \right)^{it} u_{kl}^{\alpha},\]
\[ \sigma_t(u_{kl}^{\alpha}) = \left( q_k(\alpha)q_l(\alpha) \right)^{it} u_{kl}^{\alpha};\]
in addition we have a normalization $\sum_{k=1}^{n_{\alpha}} q_k(\alpha)= \sum_{k=1}^{n_{\alpha}} q_k(\alpha)^{-1}$.
Then the space $\widetilde{\Com} ( \Alg)$ is spanned by (linearly independent!) elements of the form
\[ v_{kl}^{\alpha}:= \sum_{m=1}^{n_{\alpha}} u_{km}^{\alpha} \ot \wt{u_{ml}^{\alpha}}, \;\;\; \alpha\in \Irr (\QG),  k,l=1,\ldots, n_{\alpha}.\]
We have  for each $t \in \br$
\[ \wt{\sigma_t}(v_{kl}^{\alpha}) = \left( q_k(\alpha)q_m(\alpha)^2 q_l(\alpha) \right)^{it}\sum_{m=1}^{n_{\alpha}} u_{km}^{\alpha} \ot \wt{u_{ml}^{\alpha}}. \]
By linear independence of individual elements $u_{ij}^{\alpha}$ inside $\Alg$ the only possibility for $\wt{\sigma_t}(v_{kl}^{\alpha})$ to belong to $\widetilde{\Com} ( \Alg)$ is then equality $\wt{\sigma_t}(v_{kl}^{\alpha}) = \lambda_t v_{kl}^{\alpha}$ for some $\lambda_t \in \bc$. This in turn, using linear independence once again, requires that $q_m(\alpha)^{2it}$ does not depend on $m$ for each fixed $t \in \br$. This is possible if and only if in fact $q_m(\alpha)$ are constant as functions of $m$. This fact however forces them all to be equal to $1$, the modular group of $h$ trivialise and $\QG$ must be of Kac type.

\end{proof}

\subsection{Classical interpretation, alternative pictures} \label{classical}

In this subsection we describe how the considerations above should be interpreted in the classical context.
If $G$ is a compact group, then there is a natural isomorphic identification between the Hopf $^*$-algebras $\Alg$ and $\Alg^{\cop}$, implemented by the map $\check{}:\Pol(G) \to \Pol(G^{\cop})$ (which can be naturally defined also on the level of $\C(G)$ or $L^{\infty}(G)$). It is defined by the formula
\[ \check{f} = \widetilde{Sf}, \;\;\;f \in \Pol(G),\]
Thus if we transport the concepts from the last subsection to the framework involving only $\Alg = \Pol(G)$, defining
\[ \check{\phi}_D = \phi_D \circ ( \id \ot \,\check{}\, ): \Alg \odot \Alg \to \bc,\]
\[ \check{\Com} = ( \id_{\Alg} \ot(\,\check{}\,)^{-1} )\circ \widetilde{\Com}: \Alg \to \Alg \odot \Alg\]
\[ E_{\check{\phi}_D}= (\id_{\Alg \odot \Alg} \ot \check{\phi_D}) \circ \Com^{2}: \Alg \odot \Alg \to \Alg \odot \Alg,\]
\[\check{\gamma}= \check{\Com}^{-1} \circ E_{\check{\phi}_D}  \]
where $\Com^{2}: \Alg \odot \Alg \to \Alg \odot \Alg \odot \Alg \odot \Alg$ is the usual `Cartesian' coproduct, i.e.\ $\Com^{2}= \Sigma_{23} \circ (\Com \ot \Com)$,
we immediately see that, as $(\,\check{}\,)^{-1} \circ \, \widetilde{}\,=S$,
\[ \check{\phi}_D (f_1 \ot f_2) = \phi_D(f_1 \ot \wt{S f_2}) =  h(f_1 S^2f_2)=h(f_1 f_2) = \int_G f_1(g) f_2(g) dg,\;\;\; f_1, f_2 \in \Pol (G),\]
and $E_{\check{\phi}_D}$ corresponds to the averaging of functions on $G \times G$ over the diagonal:
\[ (E_{\check{\phi}_D}(f_1 \ot f_2))(g_1, g_2) = \int_G f_1(g_1g) f_2(g_2g) dg,\;\;\; f_1, f_2 \in \Pol (G), g_1, g_2 \in G,\]
\[\check{\Com} =  (\id_{\Alg} \ot S)\circ \Com.\]
Finally we have for the map introduced in Corollary \ref{gammamap}
\begin{align*} \check{\gamma} &= \check{\Com}^{-1} \circ E_{\check{\phi}_D}  = \wt{\Com}^{-1}\circ  (\id_{\Alg} \ot \,\check{}\,) \circ  (\id_{\Alg \ot \Alg} \ot \check{\phi}_D) \circ \Com^{2}
\\&= \wt{\Com}^{-1}\circ  (\id_{\Alg} \ot \,\check{}\,) \circ  (\id_{\Alg \ot \Alg}\ot \phi_D \circ ( \id \ot \,\check{}\, ))\circ \Com^{2}
\\&=  \wt{\Com}^{-1}\circ (\id_{\Alg \odot \Alg^{\cop}} \ot \phi_D) \circ (\id_{\Alg} \ot \,\check{}\, \ot \id_{\Alg} \ot \,\check{}\,) \circ \Com^{2}
\\& = \wt{\Com}^{-1}\circ (\id_{\Alg \odot \Alg^{\cop}} \ot \phi_D)  \circ \ComGG \circ (\id_{\Alg} \ot \,\check{}\,) = \gamma \circ (\id_{\Alg} \ot \,\check{}\,), \end{align*}
so for all $a, b \in \Alg$
\[\check{\gamma} (a\ot b) = \gamma (a \ot \check{b})= \gamma (a \ot \wt{Sb})=a\star Sb,\]
or more explicitly
\[ (\check{\gamma} (f_1 \ot f_2)) (g) = \int_{G} f_1(g_1) f_2(g^{-1}g_1) dg_1 = f_1*\check{f_2}(g).\]

Thus, we recover the map $\Gamma$ from \eqref{eq-Gammamap}, i.e. $\check{\gamma} = \Gamma$ on $\Alg$.

\subsection{Hierarchy of properties and beyond the Kac case}

In view of the arguments presented in this section we obtain the following hierarchy of properties:

\begin{itemize}
\item if $G$ is an abelian compact group, then we naturally have the diagonal subgroup $G_{diag}$, which is a normal subgroup of $G \times G$;
\item if $G$ is a classical, but not abelian compact group, we still have the diagonal subgroup $G_{diag}\subset G \times G$, but it is no longer normal;
\item if $\QG$ is a compact quantum group of Kac type, then we have a natural idempotent state $\phi_D$ on $\QG \times \QG^{\cop}$ which classically corresponds to integrating over the diagonal subgroup, and the associated (algebraic) homogeneous space of $\QG$ given by the formula $\wt{\Com}(\PolQG)$, see Theorem \ref{equalspaces};
\item if $\QG$ is a compact quantum group not of Kac type, then we still have a natural (algebraic) homogeneous space of $\QG$ given by the formula $\wt{\Com}(\PolQG)$, but it is no longer expected -- in other words, we do not have an associated idempotent state.
\end{itemize}

Note that as explained in Section \ref{Notations} we have a natural isomorphism of compact quantum groups $\QG^{\cop}$ and  $\QG^{\op}$.

\begin{rem}
When $\QG$ is not Kac, then $\phi_D$ is still unital and idempotent, with the same proof as in Proposition \ref{Diagonal}, but it is no longer positive. On the other hand,
\[
\phi_{D,t}(a\ot \wt{b}) = h\big(\sigma_{t+i/2}(a)Rb\big), \quad a,b\in\Alg,
\]
defines a state on $\Alg\odot\Alg^\cop$ for any $t\in\mathbb{R}$, since
\begin{align*}
\phi_{D,t}((a \ot \widetilde{b})^* &(a \ot \wt{b})) = \phi_{D,t}(a^*a \ot \widetilde{b^*b}) = h\big(\sigma_{t+i/2}(a^*a) R(b^*b)\big) \\
&= h\big(\sigma_{t+i/2}(a^*) \sigma_{t+i/2}(a) R(b) R(b)^*\big) = h\big(\sigma_{t+i/2}(a) R(b) R(b)^* \sigma_{t-i/2}(a^*) \big)\\
& = h\Big(  \sigma_{t+i/2}(a)R(b)\big( \sigma_{t+i/2}(a)R(b)\big)^*\Big) \geq 0.
\end{align*}
But $\phi_{D,t}$ is never idempotent. It is clear from Proposition \ref{Kac=expected} that no idempotent state corresponding to integration over the diagonal subgroup can exist in the general case, since the corresponding homogeneous space is expected if and only if we are in the Kac case.
\end{rem}

\section{Construction of a Fourier-like algebra $\A_\Delta(\QG)$}

In this section we introduce the algebra $\A_\Delta(\QG)$ associated to a compact quantum group $\QG$, the analog of the algebra $\A_\Delta(G)$ associated to a classical compact group in \cite{FSS}, motivate the construction and show it yields a completely contractive Banach algebra. Basically we will follow the classical construction, focusing firstly on the $\PolQG$-level. Throughout the sections we refer to Subsection \ref{discreteintro} for the notations used.

\subsection{Definition of $\A_\Delta(\QG)$ as an operator space}

Let $\QG$ be a compact quantum group of Kac type. In the last section we constructed the map \eqref{gammamap}
	$$\gamma=\widetilde{\Com}^{-1} \circ E_{\phi_D}: \PolQG \odot \Pol(\QG^{\cop}) \to \PolQG,\; a \ot \wt{b} \mapsto a \star b, \;\;\;\; a, b \in \PolQG.$$
Following the classical situations we would like to extend $\gamma$ to a map
	$$\gamma: \A(\QG)\prt \A(\QG^{\cop}) \to \A(\QG)$$
and define the algebra $\A_\Delta(\QG)$ as the range of the extended $\gamma$ with the operator space structure making $\gamma$ a complete quotient map.

In order to get a norm formula like \eqref{eq-norm-A_Delta} we need a more detailed understanding of the relationship between spaces $\A(\QG^{\cop})$ and $\A(\QG)$. Note that the representation theory of $\QG^{\cop}$ is naturally isomorphic to that of $\QG$, with the correspondence implemented by the antipode. More precisely, for any representation $U = [u_{ij}]^n_{i,j=1} \in M_n(\PolQG)$ of $\QG$ we can easily check that $[S(u_{ij})]^n_{i,j=1}$ is a representation of $\QG^{\cop}$ and this correspondence is clearly invertible and preserves unitarity and irreducibility, so that we can conclude
	$$\Irr(\QG) \cong \Irr(\QG^{\cop}),\;\;\;\; U^\alpha \mapsto (I_{n_\alpha}\otimes S)(U^\alpha).$$
Moreover, we have the corresponding multiplicative unitary $U^{\cop}$ of $\QG^{\cop}$ given by
	$$U^{\cop} = \bigoplus_{\alpha \in \Irr(\QG)} \sum^{n_\alpha}_{i,j=1}S(u^{\alpha_{ij}}) \otimes e^{\alpha}_{ij}.$$
Thus, we have
	\begin{align*}
		\F^{\cop}(\widetilde{Sa})
		& = (\widetilde{Sa}\cdot h\otimes I)U^{\cop}\\
		& = \bigoplus_{\alpha \in \Irr(\QG)} \sum^{n_\alpha}_{i,j=1} h(S(u^{\alpha_{ij}}) S(a))e^{\alpha}_{ij}\\
		& = \bigoplus_{\alpha \in \Irr(\QG)} \sum^{n_\alpha}_{i,j=1} h(u^{\alpha_{ij}} a)e^{\alpha}_{ij}\\
		& = \F(a)
	\end{align*}
for any $a\in \PolQG$. Once we take the embedding \eqref{eq-embedding} into account we get the following complete isometry.
	$$\A(\QG)  \cong \A(\QG ^{\cop}),\;\; a \mapsto \widetilde{Sa}.$$
Combining this with the map $\gamma$ we get the map
	$$\Phi: \PolQG \odot \PolQG \to \PolQG,\;\; a\otimes b \mapsto a\star(Sb).$$
Following the classical situation (see Subsection \ref{classical}) we call this map the {\it twisted convolution}, which can be easily extended to the Fourier algebra level as follows (with the same symbol):
	$$\Phi: \A(\QG) \prt \A(\QG) \to \A(\QG),\;\; a\otimes b \mapsto a\star(Sb).$$
	\begin{deft}
	We define the space $\A_\Delta(\QG)$ by
		$$\A_\Delta(\QG) := \text{ran}\, \Phi \subseteq \A(\QG).$$
We endow the space $\A_\Delta(\QG)$ with the operator space structure  which makes $\Phi$ a complete quotient map.
	\end{deft}

	\begin{tw}\label{thm-A_Delta-Banach}
	We have
		$$\A_\Delta(\QG) \cong \ell^1\text{-}\bigoplus_{\alpha \in \Irr(\Gb)} n_\alpha^{3/2} (S^2_{n_\alpha})_r$$
	completely isometrically. In particular, we have (recalling the notation of Subsection \ref{discreteintro}) for all $X \in c_{00}(\hQG)$
		\begin{equation}\label{eq-A_Delta}
		\|X\cdot \widehat{h} \|_{\A_\Delta(\QG)} = \sum_\alpha n^{3/2}_\alpha \|X_\alpha \|_{S^2_{n_\alpha}}.
		\end{equation}
	\end{tw}
\begin{proof}
We begin with the complete isometry
	$$\Phi^*:\A_\Delta(\QG)^* \to (A(\Gb) \prt\A(\QG))^* \cong \ell^{\infty}(\hQG)\overline{\otimes}\ell^{\infty}(\hQG).$$
Note that $c_{00}(\widehat{\QG}) \subseteq \A_\Delta(\QG)^*$ and we have (for $\alpha, \beta, \gamma \in \Irr(\QG)$ and suitable indices $i, j, k, l, p$ and $q$)
	$$\la \Phi^*(e^{\alpha}_{ij}), u^{\beta}_{kl} \otimes u^{\gamma}_{pq} \ra = \la e^{\alpha}_{ij},  u^{\beta}_{kl} \star S(u^{\gamma}_{pq})\ra.$$
Since we have
	\begin{align*}
		u^{\beta}_{kl} \star S(u^{\gamma}_{pq})
		& = u^{\beta}_{kl} \star u^{\bar{\gamma}}_{qp} \\
		& = (h\otimes I)\left[\left(S \otimes I)(\sum_s u^{\bar{\gamma}}_{qs} \otimes u^{\bar{\gamma}}_{sp})\right) \cdot (u^{\beta}_{kl} \otimes 1)\right] \\
		& = h\big( (u^{\bar{\gamma}}_{sq})^* u^\beta_{kl}\big) u^{\bar{\gamma}}_{sp} = \delta_{\beta, \bar{\gamma}} \delta_{q,l}\frac{1}{n_\beta} u^{\beta}_{kp},
	\end{align*}
where $\bar{\gamma}$ is the conjugate representation of $\gamma$, we can conclude that
	$$\la \Phi^*(e^{\alpha}_{ij}), u^{\beta}_{kl} \otimes u^{\gamma}_{pq} \ra = \frac{1}{n_\alpha}\delta_{\alpha, \beta}\delta_{\beta, \bar{\gamma}}\delta_{i,k}\delta_{j,p}\delta_{q,l},$$
which implies that
	$$\Phi^*(e^{\alpha}_{ij}) = \frac{1}{n_\alpha}\sum_{r=1}^{n_\alpha}e^{\alpha}_{ir}\otimes e^{\bar{\alpha}}_{jr}.$$
Now we note that
	\begin{align*}	\ell^\infty(\widehat{\Gb})\overline{\otimes}\ell^\infty(\widehat{\QG})
	& \cong \ell^\infty\text{-}\bigoplus_{\alpha, \beta \in \text{Irr}(\QG)}(M_{n_\alpha} \ot  M_{n_\beta})\\
	&\cong \ell^\infty\text{-}\bigoplus_{\alpha, \beta \in \Irr(\Gb)}(C_{n_\alpha} \ot  R_{n_\alpha} \ot  C_{n_\beta} \ot  R_{n_\beta}).
	\end{align*}
Let $\Sigma_{2,3}$ be the map flipping the 2nd and the 3rd tensor component at each direct summand, which is a complete isometry. We obtain then
	\begin{align*}
	\Phi^*(e^{\alpha}_{ij})
	& = \frac{1}{n_\alpha}\sum_{1\le r\le n_\alpha}e^{\alpha}_{ir}\otimes e^{\bar{\alpha}}_{jr}\\
	& \cong \frac{1}{n_\alpha}\sum_{1\le r\le n_\alpha}e^{\alpha}_i \otimes e^{\alpha}_r \otimes e^{\bar{\alpha}}_j \otimes  e^{\bar{\alpha}}_r\\
	& \stackrel{\Sigma_{2,3}}{\cong} \frac{1}{n_\alpha}\sum_{1\le r\le n_\alpha}e^{\alpha}_i \otimes e^{\bar{\alpha}}_j \otimes e^{\alpha}_r \otimes  e^{\bar{\alpha}}_r\\
	& \in \ell^\infty\text{-}\bigoplus_{\alpha\in \Irr(\Gb)}(C_{n_\alpha} \ot  C_{n_\alpha} \ot  R_{n_\alpha} \ot  R_{n_\alpha}).
	\end{align*}
Finally for a fixed $\alpha \in \Irr(\QG)$ and $X=(\delta_{\beta\alpha} X_\alpha)_{\beta \in \Irr(\QG)} \in c_{00}(\hQG)$  we have
	\begin{align*}
	\Phi^*(X)
	& = \sum_{i,j=1}^{n_\alpha} X_{\alpha}(i,j) \frac{1}{n_\alpha}\sum_{1\le r\le n_\alpha}e^{\alpha}_{ir}\otimes e^{\bar{\alpha}}_{jr}\\
	& \stackrel{\Sigma_{2,3}}{\cong} (\sum_{i,j=1}^{n_\alpha} X_{\alpha}(i,j) e^{\alpha}_i \otimes e^{\bar{\alpha}}_j)\otimes (\frac{1}{n_\alpha}\sum_{1\le r\le n_\alpha} e^{\alpha}_r \otimes  e^{\bar{\alpha}}_r)\\
	& = X_\alpha \otimes (\frac{1}{n_\alpha}\sum_{1\le r\le n_\alpha} e^{\alpha}_r \otimes  e^{\bar{\alpha}}_r).
	\end{align*}
Since we have
	$$||\frac{1}{n_\alpha}\sum_{1\le r\le n_\alpha} e^{\alpha}_r \otimes  e^{\bar{\alpha}}_r||_{R_{n_\alpha} \ot  R_{n_\alpha}} = n_\alpha^{-\frac{1}{2}}$$
we get
	$$\A_\Delta(\QG)^* \cong \ell^\infty\text{-}\oplus_{\alpha \in \Irr(\Gb)} n_\alpha^{-1/2} (S^2_{n_\alpha})_c$$
completely isometrically.	By the duality we get the conclusions we wanted.
\end{proof}

\subsection{$(\A_\Delta(\QG), \hat{\Delta}_*)$ is a completely contractive Banach algebra}
In this subsection we show that $(A_\Delta(\QG), \hat{\Delta}_*)$ is a completely contractive Banach algebra for any compact quantum group $\Gb$ of Kac type. Here,  $\hat{\Delta}$ is the co-multiplication of $L^\infty(\widehat{\Gb})$ and $\hat{\Delta}_*$ is defined only on $\PolQG \odot \PolQG$ at first, but will be extended to $A_\Delta(\QG) \prt\A_\Delta(\QG)$ later, with a slight abuse of notation. In order to show the above it is enough to check that
	$$\hat{\Delta} :\A_\Delta(\QG)^* \to (A_\Delta(\QG)\prt\A_\Delta(\QG))^*$$
is a complete contraction, which requires detailed understanding of $\hat{\Delta}$ in the matrix setting.

The next result is likely well-known to the experts. The following proof is due to Piotr So\l tan.
\begin{prop}\label{prop1}
	Let $\Gb$ be a general compact quantum group and $X = (X_\alpha)_{\alpha\in\Irr(\QG)} \in \ell^\infty(\hat{\mathbb{G}})$. Then for each $\beta, \gamma \in \Irr(\QG)$ we have the following unitary equivalence
		$$\hat{\Delta}(X)(\beta, \gamma) \cong \bigoplus_{\alpha \subseteq  \beta\ot\gamma} X_\alpha,$$
	where $\alpha \subseteq \beta\ot\gamma$ means that $u^{\alpha}$ appears in the irreducible decomposition of $u^{\beta} \otimes u^{\gamma}$.
	\end{prop}
\begin{proof}
For $\beta,\gamma\in\Irr(\QG)$ let $V(\beta,\gamma) : \bigoplus_{\alpha\subseteq \beta\ot\gamma} H_\alpha \to H_\beta \otimes H_{\gamma}$ be the unitary map providing  the irreducible decomposition of $u^{\beta} \otimes u^{\gamma}$, in our choice of representatives. Here $H_\beta$ refers to the Hilbert space associated to the representation $\beta$, i.e. $u^\beta \in B(H_\beta)\ot \PolQG$. This means that $V$ satisfies the equation
\begin{equation}\label{eq-ClG-decomp}
u^\beta\ot u^\gamma = V \left( \bigoplus_{\alpha \subseteq  \beta\ot\gamma} u^\alpha \right) V^*.
\end{equation}
Explicitely, in terms of the coefficients of the irreducible representations, this means that the $n_\beta n_\gamma\times n_\beta n_\gamma$-matrix
\[
V(\beta,\gamma) = \big( V(\beta,\gamma)^{\alpha,n)}_{k,p}\big)^{\alpha\subseteq\beta\otimes\gamma,1\le n\le n_\alpha}_{1\le k\le n_\beta,1\le p\le n_\gamma}
\]
is unitary and
\begin{equation}\label{eq-decomp}
u^\beta_{kl} u^\gamma_{pq} = \sum_{\alpha\subseteq\beta\otimes\gamma} \sum_{1\le n,m\le n_\alpha} V(\beta,\gamma)_{kp}^{\alpha,n} e^\alpha_{nm} (V(\beta,\gamma)^{\alpha,m}_{lq})^*
\end{equation}
for $1\le k,l\le n_\beta$, $1\le p,q\le n_\gamma$. This is a square matrix, since
	$$\dim \oplus_{\alpha\subseteq \beta\ot\gamma} H_\alpha = \sum_{\alpha\subseteq\beta\ot\gamma} n_\alpha = n_\beta n_\gamma = \dim H_\beta\ot H_\gamma.$$

We can write the multiplicative unitary as $U= \sum_{\alpha\in\Irr(\QG)} \sum^{n_\alpha}_{i,j=1}e^{\alpha}_{ij} \otimes u^{\alpha}_{ij}$. From the equality \eqref{comU}
we get
	$$ \sum_{\alpha\in\Irr(\QG)} \sum^{n_\alpha}_{i,j=1}\hat{\Delta}(e^{\alpha}_{ij}) \otimes u^{\alpha}_{ij} =\sum_{\beta,\gamma\in\Irr(\QG)} \sum^{n_\beta}_{k,l=1}\sum^{n_{\gamma}}_{p,q=1}e^{\beta}_{kl} \otimes e^{\gamma}_{pq} \otimes u^{\beta}_{kl}u^{\gamma}_{pq}.$$

Substituting Equation \eqref{eq-decomp} and using the linear independence of the $u^\alpha_{ij}$, we get
\[
\hat{\Delta}(e^\alpha_{ij}) = \sum_{\beta,\gamma\in\Irr(\QG)} \sum_{k,l=1}^{n_\beta}\sum_{p,q=1}^{n_\gamma} V(\beta,\gamma)^{\alpha,i}_{kp} e^\beta_{kl}\ot e^\gamma_{pq} (V(\beta,\gamma)^{\alpha,j}_{lq})^*.
\]
where we use the convention $V(\beta,\gamma)_{kp}^{\alpha,n}=0$ if $\alpha\not\subseteq\beta\otimes\gamma$.

Linearity implies that for for a general element $X= \sum_{\alpha\in\Irr(\QG)} \sum_{1\le i,j\le n_\alpha} x^\alpha_{ij} e^\alpha_{ij} \in c_{00}(\hQG)$ (which by normality suffices to complete the proof)
\[
\hat{\Delta}(X) = \sum_{\beta,\gamma\in\Irr(\QG)} \sum_{k,l=1}^{n_\beta}\sum_{p,q=1}^{n_\gamma} V(\beta,\gamma)^{\alpha,i}_{kp}  x^\alpha_{ij} (V(\beta,\gamma)^{\alpha,j}_{lq})^*e^\beta_{kl}\ot e^\gamma_{pq}.
\]
Extracting the $(\beta,\gamma)$-component of $\hat{\Delta}(X) = \big(\hat{\Delta}(X)(\beta,\gamma)\big)_{\beta,\gamma\in\Irr(\QG)}\in \prod_{\beta,\gamma\in\Irr(\QG)} M_{n_\beta}\ot M_{n_\gamma}$, we get the desired formula
\[
\hat{\Delta}(\beta,\gamma) = V(\beta,\gamma) \left(\oplus_{\alpha\subseteq\beta\ot\gamma} X^\alpha \right) V(\beta,\gamma)^*.
\]
\end{proof}	

\begin{rem}
The above proposition could serve as an explanation of the existence of the extension of $\hh{\Com}$ mentioned in \eqref{eq-unbdd-extension-coprod}.
\end{rem}

	\begin{tw} \label{KacFourier}
	The pair $(A_\Delta(\QG), \hat{\Delta}_*)$ is a completely contractive Banach algebra.
	\end{tw}
\begin{proof}
The proof is essentially the same as \cite[Theorem 1.1, Theorem 1.11]{LSS}. We include the proof for the convenience of the readers. For simplicity we only show that
	$$\hat{\Delta} :\A_\Delta(\QG)^* \to (A_\Delta(\QG)\prt\A_\Delta(\QG))^*$$
is a contraction. We have,  considering  $X=( X_\alpha)_{\alpha \in \Irr(\QG)} \in c_{00}(\hQG)$
	\begin{align*}
	\lefteqn{||\hat{\Delta}(X)||_{(A_\Delta(\QG)\prt\A_\Delta(\QG))^*} }\\
	&= \sup_{\beta, \gamma\in \Irr{\QG}}n^{-1/2}_\beta n^{-1/2}_\gamma ||\hat{\Delta}(X)(\beta, \gamma)||_{(S^2_{n_\beta})_c \ot  (S^2_{n_\gamma})_c}
	\end{align*}
and by Proposition \ref{prop1}
	\begin{align*}
	||\hat{\Delta}(X)(\beta, \gamma)||_{(S^2_{n_\beta})_c \ot  (S^2_{n_\gamma})_c}
	& = ||V(\oplus_{\delta \subseteq \beta \otimes \gamma} X_\delta)V^* ||_{(S^2_{n_\beta})_c \ot  (S^2_{n_\gamma})_c}\\
	& = (\sum_{\delta \subseteq \beta \otimes \gamma} ||X_\delta||^2_{S^2_{n_\delta}})^{1/2}\\
	& \le (\sum_{\delta \subseteq \beta \otimes \gamma} n_\delta)^{1/2}||X||_{A_\Delta(\QG)^*}.
	\end{align*}
Since $\sum_{\delta \subseteq \beta \otimes \gamma} n_\delta = n_\beta n_\gamma$ we get the desired conclusion. The matrix case calculation is similar.
\end{proof}

\subsection{The case of non-Kac type compact quantum groups}

When the compact quantum group $\QG$ is of non-Kac type, we can follow the final norm formula of \eqref{eq-A_Delta} to propose a definition of the norm on $\A_\Delta(\QG)$ as follows:
	$$\|X\cdot \widehat{h} \|_{A^2} = \sum_\alpha d^{3/2}_\alpha \| X_\alpha Q^{1/2}_\alpha \|_{S^2_{n_\alpha}}.$$
The operator space structure of $\A_\Delta(\QG)$ will be then similarly given by $\ell^1\text{-}\oplus_{\alpha \in \text{Irr}(\Gb)} d_\alpha^{3/2} (S^2_{n_\alpha})_r$. Note that the classical dimensions of representations, ($n_{\alpha}$), have been replaced by quantum ones, ($d_{\alpha}$).

\begin{rem}\label{rem-VQ}
Let $\alpha, \beta \in \Irr(\QG)$ and let $V: \bigoplus_{u^{\gamma} \subset u^{\alpha} \otimes u^{\beta}} H_\gamma \to H_\alpha \otimes H_\beta$ be the unitary map arising via the irreducible decomposition of $u^{\alpha} \otimes u^{\beta}$. Then we also have corresponding relation of $Q$-matrices:
	$$Q_\alpha \otimes Q_\beta = V(\bigoplus_{u^{\gamma} \subset u^{\alpha} \otimes u^{\beta}} Q_\gamma )V^*.$$
It follows essentially by the uniqueness of the matrices $Q_{\alpha}$ (see \cite{worLH}).
\end{rem}

We can now follow the same argument as in Theorem \ref{KacFourier}.

	\begin{tw}
	$(A_\Delta(\QG), \hat{\Delta}_*)$ is a completely contractive Banach algebra for a general compact quantum group $\QG$.
	\end{tw}
\begin{proof}
For simplicity we only show that
	$$\hat{\Delta} :\A_\Delta(\QG)^* \to (\A_\Delta(\QG)\prt\A_\Delta(\QG))^*$$
is a contraction. Now we have, again for $X=(X_{\alpha})_{\alpha \in \Irr(\QG)} \in c_{00}(\hQG)$ and each $\alpha \in \Irr(\QG)$
	\begin{align*}
	\lefteqn{||\hat{\Delta}(X)||_{(A_\Delta(\QG)\prt\A_\Delta(\QG))^*} }\\
	&= \sup_{\beta, \gamma\in \Irr(\QG) }d^{-1/2}_\beta d^{-1/2}_\gamma ||(Q^{1/2}_\beta \otimes Q^{1/2}_\gamma)\hat{\Delta}(X)(\beta, \gamma)||_{(S^2_{n_\beta})_c \ot  (S^2_{n_\gamma})_c}
	\end{align*}
and by Proposition \ref{prop1} and Remark \ref{rem-VQ}
	\begin{align*}
	||\hat{\Delta}((X_\alpha))(\beta, \gamma)||_{(S^2_{n_\beta})_c \ot  (S^2_{n_\gamma})_c}
	& = ||V(\oplus_{\delta \subseteq \beta \otimes \gamma}Q^{1/2}_\delta X_\delta)V^* ||_{(S^2_{n_\beta})_c \ot  (S^2_{n_\gamma})_c}\\
	& = (\sum_{\delta \subseteq \beta \otimes \gamma} ||Q^{1/2}_\delta X_\delta||^2_{S^2_{n_\delta}})^{1/2}\\
	& \le (\sum_{\delta \subseteq \beta \otimes \gamma} d_\delta)^{1/2}||X||_{A_\Delta(\QG)^*}.
	\end{align*}
Since $\sum_{\delta \subseteq \beta \otimes \gamma} d_\delta = d_\beta d_\gamma$ we get the desired conclusion. The matrix case calculation is similar.
\end{proof}

\begin{rem}
The last result shows we still get a completely contractive algebra $A_\Delta(\QG)$ for a non-Kac type $\QG$, but there is no longer any apparent connection between  $A_\Delta(\QG)$ and ``the diagonal subgroup of $\QG\times \QG$''.
\end{rem}

\section{Non-operator weak amenability of $\A_\Delta(O^+_N)$}

In this section we discuss general strategies for showing that the (operator) Banach algebra $\A_\Delta(\Gb)$ associated to a compact quantum group $\QG$ of Kac type  does not satisfy the weakest possible form of amenability, namely operator weak amenability, and implement this in showing that $\A_\Delta(O^+_N)$ is not operator weakly amenable for $N\ge2$.

\subsection{A general strategy for constructing derivations}

Recall again (see Subsection \ref{subsec-A_Delta}) that a completely contractive Banach algebra $\Alg$ is called operator weakly amenable if every completely bounded derivation $D:\Alg \to \Alg^*$ is inner.

Let $\QG$ be a compact quantum group of Kac type. We follow the idea  for the proof of non-amenability of $\A(SO(3))$ from \cite{BJohns}, where Johnson constructed a certain non-inner bounded derivation of the said Fourier algebra using Lie derivatives on $SO(3)$.  We need however an appropriate translation of these classical concepts into the quantum group language. We will split the construction of the derivation into two stages. The first stage is to consider a candidate for the derivation on Pol$(\Gb)$-level with respect to the $\A_\Delta(\Gb)$ multiplication. Recall that differential operators (with constant coefficients) in the classical case are nothing but Fourier multipliers, so that also here we begin with considering certain Fourier multipliers on Pol$(\Gb)$. For $A = (A_\alpha)\in \prod_{\alpha}M_{n_\alpha}$ (to be called the \emph{symbol of the multiplier $T_A$}) we set
	$$T_A : \PolQG \to \PolQG,\; a \mapsto \F^{-1}(A\F(a))$$
or equivalently
	\begin{equation}\label{eq-multiplier}
	T_A(B\cdot \widehat{h}) = AB \cdot \widehat{h},
	\end{equation}
where $B\in c_{00}(\hQG)$. Note that to define the map $T_A$ on the $\PolQG$-level we do not need uniform boundedness of the symbol $A$, and we will actually use unbounded symbols for our construction below.
   	
Note the (algebraic) adjoint map $T^*_A$ is given by
	\begin{equation}\label{eq-adj-T_A}
	T^*_A(C) = CA, \; C\in \prod_{\alpha \in \text{Irr}(\Gb)}M_{n_\alpha}.
	\end{equation}
We want the map $T_A$ to be a derivation with respect to $A_\Delta(\QG)$-multiplication, which is equivalent to the equality
	$$T_A\circ \hat{\Delta}_* = \hat{\Delta}_* \circ (I \otimes T_A) + \hat{\Delta}_* \circ (T_A\otimes I).$$
By taking adjoint and recalling \eqref{eq-adj-T_A} we can conclude that we need
	\begin{equation}\label{eq-unbdd-point-der}
	\hat{\Delta}(A) = A\otimes I + I \otimes A,
	\end{equation}
where $\hat{\Delta}$ is the natural extension of the coproduct of $\hQG$ to unbounded sequences (and the equality above can be understood formally).

For the second stage of the construction, we consider the following natural map
	\begin{equation}\label{eq-Phi}
	\Psi:\A_\Delta(\QG) \to (A_\Delta(\QG))^*,\; B\cdot \widehat{h} \to \widehat{S}(B),
	\end{equation}
where $\widehat{S}$ is the antipode of the dual quantum group $\widehat{\Gb}$. Note that we can easily see that $\Psi$ is a contraction. Indeed,  by \eqref{eq-antipode-dual-transpose}, Theorem \ref{thm-A_Delta-Banach} and the fact that $n_\alpha = n_{\bar{\alpha}}$, $\alpha\in \Irr(\QG)$ we have
	$$||\widehat{S}(B)||_{(A_\Delta(\QG))^*} = \sup_{\alpha}n^{-\frac{1}{2}}_\alpha ||B^t_{\bar{\alpha}}||_{S^2_{n_\alpha}}\le \sum_{\alpha}n^{\frac{3}{2}}_\alpha ||B_\alpha||_{S^2_{n_\alpha}} = ||B\cdot \widehat{h}||_{\A_\Delta(\QG)}.$$
Moreover, the map $\Psi$ is an $A_\Delta(\QG)$-bimodule map. Indeed, we can easily check that the adjoint map $\Psi^*$ is given by
	$$\Psi^*(C\cdot \widehat{h}) = \widehat{S}(C), \; C\cdot \widehat{h}\in\A_\Delta(\QG)\subset \A_\Delta(\QG)^{**},$$	
so that we have
	\begin{align*}
		\la \Psi\circ \hat{\Delta}_*(A\cdot \widehat{h} \otimes B\cdot \widehat{h}),  C\cdot \widehat{h}\ra
		& = (A\cdot \widehat{h} \otimes B\cdot \widehat{h})( \hat{\Delta}(\widehat{S}(C))).
	\end{align*}
At the same time we have
	\begin{align*}
		\la \hat{\Delta}^L_*(A\cdot \widehat{h} \otimes & \Psi(B\cdot \widehat{h})), C\cdot \widehat{h}\ra
		 = \la (\widehat{S}(B),  \hat{\Delta}_*(C\cdot \widehat{h} \otimes A\cdot \widehat{h})\ra
		= (C\cdot \widehat{h} \otimes A\cdot \widehat{h})( \hat{\Delta}(\widehat{S}(B)))
		\\& = (A\cdot \widehat{h} \otimes C\cdot \widehat{h})( (\widehat{S}\otimes \widehat{S})\hat{\Delta}(B))
 =  (\widehat{h} \ot  \widehat{h}) ((\widehat{S}\otimes \widehat{S})\hat{\Delta}(B) (A \ot C))
 \\&= (\widehat{h}\circ \widehat{S} \ot  \widehat{h}\circ \widehat{S}) ( (\widehat{S}(A) \ot \widehat{S}(C)) \hat{\Delta}(B))
= 		(\widehat{h} \ot  \widehat{h}) ( (\widehat{S}(A) \ot \widehat{S}(C)) \hat{\Delta}(B))
\\&= (\widehat{h}\cdot \widehat{S}(A)) ((\id \ot \widehat{h})((1 \ot \widehat{S}(C)) \hat{\Delta}(B)))
=  (\widehat{h}\cdot \widehat{S}(A)) (\widehat{S} ((\id \ot \widehat{h}) \hat{\Delta}(\widehat{S}(C)) (1 \ot B)))
  \\&= (\widehat{h} \circ \widehat{S}) (((\id \ot \widehat{h}) \hat{\Delta}(\widehat{S}(C)) (1 \ot B))A) = (A\cdot \widehat{h} \otimes B\cdot \widehat{h})( \hat{\Delta}(\widehat{S}(C))) ,
	\end{align*}
where 	$\hat{\Delta}^L_*$ denoted the left module action of $A_\Delta(\QG)$ on $A_\Delta(\QG)^*$. Note that we use the fact that $\widehat{S}$ (or rather its natural extension to unbounded elements) is an involution and that it preserves the Haar weight. Finally  the second to the last equality follows from the strong left invariance. Combining the above two equations we can conclude that $\Psi$ is a left $A_\Delta(\QG)$-module map. The right  $A_\Delta(\QG)$-module property can be shown similarly.

The composed map
	$$\Psi \circ T_A: \PolQG \subseteq\A_\Delta(\QG) \to\A_\Delta(\QG)^*$$
clearly satisfies, in this case, the derivation condition for $\A_\Delta(\QG)$ multiplication. Thus we hope to find a symbol $A = (A_\alpha)_{\alpha \in \Irr (\QG)}\in \prod_{\alpha\in \Irr (\QG)}M_{n_\alpha}$ satisfying the condition \eqref{eq-unbdd-point-der} and such that $\Psi \circ T_A$ can be extended to a completely bounded map from $\A_\Delta(\QG)$ into $\A_\Delta(\QG)^*$. Finally, we would then need to check that the resulting derivation $\Psi \circ T_A$ is not inner.

\subsection{Constructing derivations for compact quantum groups with classical subgroups}

We continue using the same notation as above. In order to find a symbol $A$ satisfying the condition \eqref{eq-unbdd-point-der} we could begin with a weak $^*$-continuous one-parameter group of group-like unitary elements $C_t \in \ell^{\infty}(\hQG)$, $t\in \br$. Recall that $X\in \ell^{\infty}(\hQG)$ is called group-like if $\hat{\Delta}(X) = X \otimes X$. Taking derivative of $C_t$ at $t=0$ will give us the desired symbol $A$. If $\QG$ admits a classical closed quantum subgroup $\QH$, say isomorphic to the circle $\Tb$, we can `transport' such a family from $\QH$. Let then indeed $\QH$ be a closed quantum subgroup of $\QG$ with the canonical $*$-homomorphisms $\pi: \text{Pol}(\Gb) \to \text{Pol}(\Hb)$ and $\iota: \ell^{\infty}(\widehat{\QH}) \to \ell^{\infty}(\hQG)$ from Subsection \ref{subsection-closedQsubgroups}. Our strategy is based on using  a one-parameter group of group-like elements $D_t \in \ell^{\infty}(\widehat{\QH})$ to define the group-like elements $C_t = \iota(D_t)\in \ell^{\infty}(\hQG)$. Since $\Hb$ is a classical compact group we know that, in fact, any group-like element of $\ell^{\infty}(\widehat{\QH})$ can be identified with $\lambda_\Hb(h)$ for some $h\in \Hb$, where $\lambda_\Hb$ is the left regular representation of $\Hb$ (\cite{EnochSchwarz}). If $\QH=\Tb$ we can simply put:
	\begin{equation}\label{eq-torus-grouplike}
	D_t:=\lambda_\Tb(e^{it}) \cong (e^{itn})_{n\in \z} \in \ell^{\infty}(\z)
	\end{equation}
for $t\in \br$.

The next step would be then to check the complete boundedness of the resulting map $\Psi \circ T_A: \PolQG \subseteq\A_\Delta(\QG) \to\A_\Delta(\QG)^*.$ Since $A_\Delta(\QG) \cong \ell^1\text{-}\oplus_{\alpha \in \text{Irr}(\Gb)} n^{3/2}_\alpha (S^2_{n_\alpha})_r$ and $A_\Delta(\QG)^* \cong \ell^\infty\text{-}\oplus_{\alpha \in \text{Irr}(\Gb)} n^{-1/2}_\alpha (S^2_{n_\alpha})_c$ it suffices to verify  the uniform boundedess of the norms of $\Psi \circ T_A$ restricted to each block of the direct sum, i.e.\
	$$(\Psi \circ T_A)|_{n^{3/2}_\alpha S^2_{n_\alpha}}: n^{3/2}_\alpha (S^2_{n_\alpha})_r \to n^{-1/2}_{\bar{\alpha}} (S^2_{n_{\bar{\alpha}}})_c,\; X \mapsto (A_{\bar{\alpha}} X)^t = X^t A^t_{\bar{\alpha}}$$
for each $\alpha \in \text{Irr}(\Gb)$, taking \eqref{eq-antipode-dual-transpose} into account. Since
	$$n^{-1/2}_{\bar{\alpha}} || X^t A^t_{\bar{\alpha}} ||_{S^2_{n_{\bar{\alpha}}}} \le n^{-2}_\alpha||A_\alpha||_\infty n^{3/2}_\alpha||X||_{S^2_{n_\alpha}}$$
we have $||(\Psi \circ T_A)|_{n^{3/2}_\alpha \cdot S^2_{n_\alpha}}|| \le n^{-2}_\alpha||A_\alpha||_\infty$ and further
	$$||(\Psi \circ T_A)|_{n^{3/2}_\alpha (S^2_{n_\alpha})_r}||_{cb} \le n^{-1}_\alpha||A_\alpha||_\infty.$$
Here we used the fact that $CB(R_n, C_n) \cong S^2_n$ isometrically so that
	$$||u||_{cb} \le \sqrt{n}||u||$$
for $u\in CB(R_n, C_n)$. Thus we have
	$$||\Psi \circ T_A||_{cb} \le \sup_\alpha n^{-1}_\alpha||A_\alpha||_\infty.$$

The above tells us that we need to understand the norm growth of $A_\alpha$, which can be done by examining associated functionals on $\PolQG$. More precisely, there are uniquely determined states
$\psi_t : \text{Pol}(\Hb) \to \Comp$ associated to $D_t$, i.e. the functionals satisfying the formulas
	$$(I\otimes \psi_t)U_{\Hb} = D_t, \;\;\;\; t \in \br.$$
Then the states $\varphi_t := \psi_t \circ \pi: \PolQG \to \Comp$ are associated in the analogous way  to the elements $C_t = \iota(D_t) \in \ell^{\infty}(\hQG)$ since $C_t = (\iota \otimes \psi_t)U_\Hb = (I\otimes \psi_t\circ \pi)U_\Gb = (I\otimes \varphi_t)U_\Gb.$
Thus our choice of symbol $A = \frac{dC_t}{dt}|_{t = 0}$ is associated to the functional
	$$\phi = \frac{d\varphi_t}{dt}|_{t = 0}: \PolQG \to \Comp.$$
Note that the correspondence between $A$ and $\phi$ (which is nothing but another incarnation of the Fourier transform) means that
	$$A_\alpha = (I_{d_\alpha}\otimes \phi)(u^{\alpha}) \;\;(\text{and we may simply write in what follows $\phi(u^{\alpha})$})$$
and the same holds for $(D_t, \psi_t)$ and $(C_t, \varphi_t)$. Note further that the functionals $\psi_t$ and $\varphi_t$ are in fact characters on $\PolQG$ and $\Pol(\Hb)$, respectively, since their associated elements $C_t$ and $D_t$ are group-like.

All the above construction lead to the following completely bounded derivations.

	\begin{prop}\label{prop-cb-derivation}
	Let $\QG = O^+_N$, $N\ge 2$ and $\Hb$ be a circle group contained in $O_N$ viewed as a closed quantum subgroup of $\QG$. Let $A$ be the symbol defined as above. Then the map $\Psi\circ T_A$ extends to a completely bounded derivation
		$$\Psi \circ T_A: \A_\Delta(\QG) \to\A_\Delta(\QG)^*.$$
	\end{prop}
\begin{proof}
From the discussion above it follows that it suffices to show that $\sup_\alpha n^{-1}_\alpha||A_\alpha||_\infty <\infty$. As $\Gb = O^+_N$, $N\ge 2$,  we have $\Irr(\Gb) = \n$ (\cite{Banica}) and we will write $d_n = \text{dim}\,u^{(n)}$ for $n \in \bn$ (recall that these values are given by Tchebyshev polynomials). The fusion rules for $O^+_N$ say
	$$u^{1} \otimes u^{n} \cong u^{n+1} \oplus u^{n-1},\; n\ge 1,$$	
which 	can be transferred, through the multiplicative functionals $\varphi_t$, to unitary equivalence of scalar matrices
	$$\varphi_t(u^{1})\otimes \varphi_t(u^{n}) = \varphi_t(u^{1} \otimes u^{n}) \cong \varphi_t(u^{n+1}) \oplus \varphi_t(u^{n-1}),\; n\ge 1.$$
By taking derivative of the last formula at $t=0$ we have
	\begin{equation}\label{eq-fusion}
	\phi(u^{1})\otimes \varphi_0(u^{n}) + \varphi_0(u^{1}) \otimes \phi(u^{n}) \cong \phi(u^{n+1}) \oplus \phi(u^{n-1}),\; n\ge 1.
	\end{equation}
Note that we know that each $\varphi_t$ is a state, so in particular a complete contraction, again for each $n \in \bn$
	$$||\varphi_0(u^{n})||_{M_{d_n}} \le 1.$$
Thus, \eqref{eq-fusion} tells us that
	$$||\phi(u^{n})||_{M_{d_n}} \le n ||\phi(u^{1})||_{M_{d_1}},\; n\ge 1,$$
so that we have
	$$\sup_\alpha n^{-1}_\alpha||A_\alpha||_\infty = \sup_{n\ge 1}\frac{||\phi(u^{n})||_{M_{d_n}}}{d_n} \le \sup_{n\ge 1}\frac{n}{d_n} ||\phi(u^{1})||_{M_{d_1}} <\infty$$
for all $N\ge 2$.	
\end{proof}

We can actually extend the above construction to the case of any compact quantum group of Kac type containing $\Gb = O^+_2$ as a closed quantum subgroup.

	\begin{prop}\label{prop-cb-derivation-2}
	Let $\widetilde{\Gb}$  be a compact quantum groups of Kac type containing $\Gb = O^+_2$ as a closed quantum subgroup. Let $\Hb$ be the maximal torus (circle) in $O^+_2$ and $\tilde{A}$ be the symbol for $\widetilde{\QG}$ defined as before. Then $\Psi\circ T_{\tilde{A}}$ extends to a completely bounded derivation
		$$\Psi \circ T_{\tilde{A}}: \A_\Delta(\widetilde{\QG}) \to\A_\Delta(\widetilde{\QG})^*.$$
	\end{prop}
\begin{proof}
Let $\pi: \text{Pol}(\QG) \to \text{Pol}(\Hb)$, $\iota: \ell^{\infty}(\hQG) \to \ell^{\infty}(\widehat{\Hb})$, $\tilde{\pi}: \text{Pol}(\widetilde{\Gb}) \to \text{Pol}(\Gb)$ and $\tilde{\iota}: \ell^{\infty}(\hQG) \to \ell^{\infty}(\widehat{\widetilde{\Gb}})$ be the corresponding $*$-homomorphisms. By Proposition \ref{prop-cb-derivation} we already have group-like elements $C_t \in \ell^{\infty}(\hQG)$ and the respective derivative $A = (A_\alpha)_{\alpha\in \text{Irr}(\Gb)} = \frac{dC_t}{dt}|_{t=0}$ with the corresponding functionals $\varphi_t : \text{Pol}(\Gb) \to \mathbb{C}$, $\phi  : \text{Pol}(\Gb) \to \mathbb{C}$ satisfying the condition
	$$\sup_{\alpha\in \text{Irr}(\Gb)} \frac{||A_\alpha||_{M_{n_\alpha}}}{n_\alpha} = M<\infty.$$
Now we further transfer $C_t$ to $\widetilde{\Gb}$-level to get $\tilde{C}_t := \tilde{\iota}(C_t) \in \ell^{\infty}(\widehat{\widetilde{\Gb}})$ with the respective derivative $\tilde{A} = (\tilde{A}_\beta)_{\beta\in \text{Irr}(\widetilde{\Gb})} = \frac{d\tilde{C}_t}{dt}|_{t=0}$. Then it is straightforward to see that the functional $\tilde{\phi}  : \text{Pol}(\widetilde{\Gb}) \to \mathbb{C}$ corresponding to $\tilde{A}$ is actually given by $\tilde{\phi} = \phi \circ \tilde{\pi}.$ Now we would like to check the condition
	$$\sup_{\beta\in \text{Irr}(\widetilde{\Gb})} \frac{||\tilde{A}_\beta||_{M_{n_\beta}}}{n_\beta}<\infty.$$
Indeed, we know that $\tilde{A}_\beta = \tilde{\phi}(v^\beta)$, where $v^\beta$ is an irreducible unitary representation associated to $\beta\in \text{Pol}(\widetilde{\Gb})$. Note that $\tilde{\pi}(v^\beta)$ is a unitary representation of $\Gb$ so that we have irreducible decomposition $\tilde{\pi}(v^\beta) \cong u^{\alpha_1} \oplus \cdots \oplus u^{\alpha_k}$. Thus, we have
	$$\frac{||\tilde{A}_\beta||_{M_{n_\beta}}}{n_\beta} = \frac{||\phi(\tilde{\pi} (v^{\beta}))||_{M_{n_\beta}}}{n_\beta} = \max_{1\le j\le k}\frac{||\phi(u^{\alpha_j})||_{M_{n_{\alpha_j}}}}{n_\beta} = \max_{1\le j\le k}\frac{||A_{\alpha_j}||_{M_{n_{\alpha_j}}}}{n_\beta} \le M \max_{1\le j\le k} \frac{n_{\alpha_j}}{n_\beta} \le M,$$
the conclusion we wanted.
\end{proof}

\subsection{Non-operator weak amenability of $\A_\Delta(O^+_N)$, $N\ge 2$}

For the non-operator weak amenability of $\A_\Delta(\QG)$ we now only need to check that there is a completely bounded derivation $D = \Psi\circ T_A$ from Proposition \ref{prop-cb-derivation} which is not inner. If it were inner, there would be $Y\in\A_\Delta(\QG)^*$ such that
	$$D(B\cdot \widehat{h}) = (B\cdot \widehat{h})\cdot Y - Y\cdot (B\cdot \widehat{h})$$
for any $B \in c_{00}(\widehat{\QG})$.

From the definition \eqref{eq-multiplier} and \eqref{eq-Phi} we have for any $B,C\in c_{00}(\widehat{\QG})$ that
	$$\la D(B\cdot \widehat{h}), C\cdot \widehat{h}\ra = \widehat{h}(\widehat{S}(AB) C) = \widehat{h}(\widehat{S}(C)AB)$$
and
	$$\la (B\cdot \widehat{h})\cdot Y, C\cdot \widehat{h}\ra
	= \la  Y, \hat{\Delta}_*(C\cdot \widehat{h} \otimes B\cdot \widehat{h})\ra = \widehat{h} \otimes \widehat{h} (\hat{\Delta}(Y) (C\otimes B)).$$
We have a similar calculation for $Y\cdot (B\cdot \widehat{h})$, so that finally
	\begin{equation}\label{eq-difference}\widehat{h}(\widehat{S}(C)AB) = \widehat{h}(\widehat{S}(B)\widehat{S}(A)C) = \widehat{h} \otimes \widehat{h} (\hat{\Delta}(Y) (C\otimes B - B\otimes C)).
	\end{equation}
We will show that the above equality does not hold for the case of free orthogonal quantum groups $\QG=O^+_N$, $N\ge 2$, defined by Wang \cite{Wangfree,Banica}, for an appropriately chosen symbol $A$. In order to do that we need a detailed understanding of the symbol $A$ and $\hat{\Delta}(Y)$, at least for the components $A_1$ and $\hat{\Delta}(Y)(1,1)$.

Let $u^1=(x_{jk})_{1\le j,k\le N}$ be the fundamental representation of $O_N^+$, where the coefficients $x_{jk}$ are the usual generators of $\mathrm{Pol}(O_N^+)$. We will choose a canonical embedding of torus $\mathbb{T}$ into $O_N^+$ described by  the $*$-homomorphism $\pi : \PolQG \to \Pol(\mathbb{T})$ given by
	$$[\pi (x_{jk})]^N_{j,k=1} = \begin{bmatrix}\cos\theta & -\sin\theta \\ \sin\theta & \cos\theta\end{bmatrix} \oplus I_{N-2} \in M_N ,$$
where $I_{N-2}$ refers to the identity matrix of size $N-2$.

\begin{lem}\label{lem-particular-der}
	Let $A$ be the symbol constructed from the torus $\Hb \cong \mathbb{T}$ inside of $O_N^+$ in the above. Then we have $A_1 = e_{21}-e_{12}\in M_N.$
	\end{lem}
\begin{proof}
Let $\iota: \ell^\infty(\widehat{\mathbb{H}}) \to \ell^\infty(\widehat{\mathbb{G}})$ be the $*$-homomorphism characterized by \eqref{eq-two-homomorphisms}. We have $(\pi \otimes I) (U_\mathbb{G}) = (\pi \otimes I) (1 \oplus [x_{jk}] \oplus \cdots)$ and at the level of $u^1$ we have
	\begin{align*}
		(\pi \otimes I) (U_\mathbb{G})_1
		& =\begin{bmatrix}\cos\theta & -\sin\theta \\ \sin\theta & \cos\theta\end{bmatrix} \oplus I_{N-2} \\
		& = \cos \theta (e_{11} + e_{22}) +  \sin\theta (e_{21} - e_{12}) + \sum^N_{i=3}e_{ii}. \nonumber
	\end{align*}
Moreover, we have
	\begin{align*}
	U_\mathbb{H} & = \sum_{n\in \z} z^n \otimes e_n = 1\otimes e_0 + z\otimes e_1 + z^{-1} \otimes e_{-1} + \cdots\\
	& = 1\otimes e_0 + \cos\theta \otimes (e_1 + e_{-1}) + \sin\theta \otimes (ie_1 - ie_{-1}) + \cdots,
	\end{align*}
so that
	$$[I\otimes \iota(U_\mathbb{H})]_1 = 1\otimes \iota(e_0)_1 + \cos\theta \otimes (\iota(e_1)_1 + \iota(e_{-1})_1) + \sin\theta \otimes (i\iota(e_1)_1 - i\iota(e_{-1})_1)+ \cdots.$$
By comparing both sides of \eqref{eq-two-homomorphisms} at the level of $u^1$ we get
	$$\iota(e_0)_1 = \sum^N_{i=3}e_{ii},\;\; i\iota(e_1)_1 - i\iota(e_{-1})_1 = e_{21} - e_{12}, \;\; \iota(e_n)_1 = 0,\;\; |n|\ge 2.$$
Now recall that $D_t$ from sec 5.2 is explicitly given by $D_t = \sum_{n\in \z}e^{itn}\otimes e_n,$
so that we have $C_t = \iota(D_t) = \sum_{n\in \z}e^{itn}\otimes \iota(e_n).$ In particular, we have $(C_t)_1 = 1\otimes \iota(e_0)_1 + e^{it} \otimes \iota(e_1)_1 + e^{-it} \otimes \iota(e_{-1})_1$, so that $A_1 = i \iota(e_1)_1 - i \iota(e_{-1})_1 = e_{21} - e_{12}.$
\end{proof}

For the component $\hat{\Delta}(Y)(1,1)$ we need to know a unitary $U : \ell^2_{N^2} = \ell^2_{N^2-1} \oplus \mathbb{C} \mapsto \ell^2_N \otimes \ell^2_N$ that implements the equivalence $u^1\otimes u^1 \cong u^2 \oplus 1$. Let us use the notation $(j,k)\equiv (j-1)N+k$, which allows us the identification $\ell^2_{N^2} \cong \ell^2_N \otimes \ell^2_N$ via the map $e_{(j,k)} \mapsto e_j \otimes e_k$, where $(e_j)$ and $e_{(j,k)}$ are the canonical basis of $\ell^2_N$ and $\ell^2_{N^2}$. Then, we can regard $U$ as an operator acting on $\ell^2_N \otimes \ell^2_N$. For simplicity we often write $jk$ instead of $(j,k)$. Now we choose $U\in B(\ell^2_N \otimes \ell^2_N) \cong M_{N^2}$ to be the operator that acts as identity on $\mathrm{span}\{e_j\otimes e_k; 1\le j,k\le N, j\not= k\}$ and as the Fourier transform of $\mathbb{Z}_n$ on $\mathrm{span}\{e_j\otimes e_j; j=1,\ldots, N\}$. The coefficients of $U$ w.r.t.\ $\{e_j\otimes e_k;1\le j,k\le N\}$ are given by
\begin{equation}\label{eq-U}
U_{jk,\ell m} =
(1-\delta_{jk})\delta_{j\ell}\delta_{km} + \delta_{jk}\delta_{\ell m} \frac{\exp(2\pi i j\ell/N)}{\sqrt{N}}
\end{equation}
for $j,k,\ell,m\in \{1,\ldots,N\}$.

\begin{prop}
With $U$ as in Equation \eqref{eq-U} and $u^1$ the fundamental representation of $O_N^+$, we have
\[
U^* u^1\otimes u^1 U = \left( \begin{array}{cc} u^2 & 0 \\ 0 & 1\end{array}\right),
\]
where $u^2$ is a representative of the second non-trivial irreducible unitary representation of $O_N^+$.
\end{prop}
\begin{proof}
It is sufficient to check the last row and column, then it follows from the representation theory of $O_N^+$ that the remaining block has to be unitarily equivalent to the unique irreducible unitary representation of dimension $N^2-1$.

We have $u^1\otimes u^1=(x_{jr}x_{ks})_{j,k,r,s=1}^N \in B(\ell^2_N \otimes \ell^2_N) \otimes \C(O_N^+) \cong  M_{N^2}(\C(O_N^+))$. Note that the pair $(j,k)$ corresponds to the rows and the pair $(r,s)$ corresponds to the columns.
Using the relations $\sum_{\ell=1}^N x_{\ell j} x_{\ell k}=\delta_{jk}=\sum_{\ell=1}^N x_{j\ell} x_{k\ell}$ we get
\begin{eqnarray*}
(U^* u^1\otimes u^1 U)_{NN,NN} &=& \sum_{j,k,r,s=1}^N \overline{U}_{jk,NN} x_{jr} x_{ks} U_{rs,NN} \\
&=& \sum_{j,r=1}^N \frac{\exp\big(2\pi i(rN-jN)/N\big)}{N} x_{jr}^2 \\
&=& \frac{1}{N}\sum_{j,r=1}^N x_{jr}^2= 1,
\end{eqnarray*}
for the entry in the lower right corner, and
\begin{eqnarray*}
(U^* u^1\otimes u^1 U)_{NN,\ell m} &=& \sum_{j,k,r,s=1}^N \overline{U}_{jk,NN} x_{jr} x_{ks} U_{rs,\ell m} \\
&=& \sum_{j,r,s=1}^N \frac{\exp\big((-2\pi i Nj)/N\big)}{\sqrt{N}} x_{jr} x_{js}  U_{rs,\ell m} \\
&=& \frac{1}{\sqrt{N}} \sum_{r,s=1}^N \delta_{rs}  U_{rs,\ell m}
= \sum_{r=1}^N \delta_{\ell m} \frac{\exp(2\pi i r\ell/N )}{N}
= 0
\end{eqnarray*}
for $(\ell,m)\not=(N,N)$.

In the same way we also have $(U^* u^1\otimes u^1 U)_{\ell m,NN} = 0$ for the off-diagonal entries of the last column.
\end{proof}

We can use $U$ for a partial computation of the dual coproduct $\hat{\Delta}$ by Proposition \ref{prop1}, so that we get the following result.
\begin{lem}\label{lem-copr-Y(1,1)}
Let $\QG = O_N^+$ and $Y = (Y_n)_{n\ge 0}\in \prod_{n\ge 0}M_{d_n}$ with
$$Y_0 = y_{N^2, N^2}\;\;\text{and}\;\; Y_2 = \begin{bmatrix} y_{11} & \cdots  & y_{1, N^2-1} \\ \vdots  & \ddots  & \vdots \\ y_{N^2-1,1} & \cdots & y_{N^2-1,N^2-1}\end{bmatrix},$$ for coefficients $y_{11},\ldots,y_{N^2-1,N^2-1}, y_{N^2, N^2} \in \bc$. We set $y_{N^2, s} = y_{t, N^2} = 0$ for $1\le s,t \le N^2-1$.
Then the coefficients of $\hat{\Delta}(Y)(1,1)$ are given by
\[
\hat{\Delta}(Y)(1,1)_{\ell m,pq} =
\left\{\begin{array}{ccl}
y_{\ell m,pq} & \mbox{ if} & \ell\not= m, p\not=q, \\
\sum_{j=1}^{N} \frac{\exp(-2\pi i j\ell/N)}{\sqrt{N}} y_{jj,pq} & \mbox{ if} & \ell= m, p\not=q, \\
\sum_{r=1}^{N} \frac{\exp(2\pi i pr/N)}{\sqrt{N}} y_{\ell m,rr} & \mbox{ if} & \ell \not= m, p=q, \\
\sum_{j,r=1}^{N} \frac{\exp\big(2\pi i(rp-j\ell)/N\big)}{N} y_{jj,rr} & \mbox{ if} & \ell= m, p=q. \\
\end{array}\right.
\]
\end{lem}

\begin{tw}
The algebra $\A_\Delta(O^+_N)$ is not operator weak amenable for $N\ge 2$.
\end{tw}
\begin{proof}	
Suppose that the completely bounded derivation $D = \Psi\circ T_A$ from Proposition \ref{prop-cb-derivation} and Lemma \ref{lem-particular-der} is inner. Then there would be $Y\in\A_\Delta(\QG)^*$ satisfying \eqref{eq-difference} for any $B,C\in c_{00}(\widehat{\QG})$. If we put $B$ and $C$ only supported on the representation $1$ in \eqref{eq-difference} we get
	\begin{align}\label{eq-final-N}
		2\text{Tr}(A_1B_1(\widehat{S}(C)_1))
		& = 4\text{Tr}\otimes \text{Tr}(\hat{\Delta}(Y)(1,1)(C_1\otimes B_1 - B_1\otimes C_1))\\
		& = 4\text{Tr}\otimes \text{Tr}((\hat{\Delta}(Y)(1,1) - \Sigma \hat{\Delta}(Y)(1,1)) (C_1\otimes B_1)).\nonumber
	\end{align}
For the left hand side we recall that $1 \cong \bar{1}$ and that the unitary matrix of the equivalence $1 \cong \bar{1}$ is the identity matrix. This implies that $((\widehat{S}(X))_1)^t$ is equal to $X_1$.

Thus, for the choices $B_1 = e_{11}\in M_N$ and $C_1 = e_{12}+e_{21}\in M_N$, we get the left hand side of \eqref{eq-final-N} being $2$.

Note that the flip map $\Sigma$ acts on matrix units via $\Sigma(e_{ij}\otimes e_{kl}) = e_{kl}\otimes e_{ij}$, so that the matrix presentation of Lemma \ref{lem-copr-Y(1,1)}  gives us

\[
\big(\hat{\Delta}(Y)(1,1)-\Sigma \hat{\Delta}(Y)(1,1)\big)_{\ell m,pq} =
\left\{\begin{array}{ccl}
y_{\ell m,pq}-y_{m\ell,qp} & \mbox{ if} & \ell\not= m, p\not=q, \\
\sum_{j=1}^{N} \frac{\exp(-2\pi i j\ell/N)}{\sqrt{N}} (y_{jj,pq}-y_{jj,qp}) & \mbox{ if} & \ell= m, p\not=q, \\
\sum_{r=1}^{N} \frac{\exp(2\pi i pr/N)}{\sqrt{N}} (y_{\ell m,rr}-y_{m\ell,rr}) & \mbox{ if} & \ell \not= m, p=q, \\
0 & \mbox{ if} & \ell= m, p=q, \\
\end{array}\right.
\]
for the coefficients of $\hat{\Delta}(Y)(1,1)-\Sigma \hat{\Delta}(Y)(1,1)$.

For the same choices $B_1 = e_{11}$ and $C_1 = e_{12}+e_{21}$ we get the right hand side of \eqref{eq-final-N} being $0$ for any $Y \in A_\Delta(\mathbb{G})^*$. Indeed, for $Z = \hat{\Delta}(Y)(1,1)-\Sigma \hat{\Delta}(Y)(1,1)$,  we have
	\begin{align*}
	\text{Tr}\otimes \text{Tr}(Z (C_1\otimes B_1)) & = \text{Tr}\otimes \text{Tr}(Z (e_{21}\otimes e_{11})) + \text{Tr}\otimes \text{Tr}(Z (e_{12}\otimes e_{11}))\\
	 & = Z_{12, 11} + Z_{21, 11} = 0.
	 \end{align*}
Thus, the derivation $D=\Psi\circ T_A$ is not inner.
\end{proof}

\begin{rem}
The non-innerness of the constructed derivation is a new technicality in the quantum situation. Note that it is automatically non-inner in the classical situation, i.e. when $\hat{\Delta}$ is symmetric. We supect all the derivations constructed in Proposition \ref{prop-cb-derivation-2} are non-inner, but are not able to provide the proof at the time of this writing.
\end{rem}

\subsection{Remarks on the non-weak amenability of $\A(O^+_N)$, $N\ge 2$}

The same approach as in the previous sections actually tells us that the `standard' Fourier algebra $\A(O^+_N)$, $N\ge 2$ is not weakly amenable, i.e.\ the same derivation can be shown to be bounded from $\A(O^+_N)$ into $\A(O^+_N)^*$, which boils down to checking the norms of the maps $n_\alpha S^1_{n_\alpha} \to M_{n_{\bar{\alpha}}},\;\; X \mapsto X^t A^t_{\bar{\alpha}}$ are uniformly bounded. We will record it here without repeating the proof.

	\begin{tw}\label{thm-NWA-O^+_2}
	The algebra $\A(O^+_N)$, $N\ge 2$ is not weakly amenable.
	\end{tw}

For the case of $O^+_N$, $N\ge 3$ we actually have a simpler way of explaining the same conclusion due to Ebrahim Samei and Nico Spronk.

	\begin{tw}\label{thm-NWA}
	Let $\Gb$ be a general compact quantum group containing a closed classical subgroup $H$ whose connected component of the identity is not abelian. Then, $\A(\Gb)$ is not weakly amenable.
	\end{tw}
\begin{proof}
Recall that the canonical homorphism $\pi : \text{Pol}(\Gb) \to \text{Pol}(H)$ can be easily transferred to the Fourier algebras (\cite[Theorem 3.7]{DawKasSkaSol}), with the corresponding map
	$$\pi_{\A} : \A(\Gb) \to \A(H)$$	
remaining surjective.
Thus, we have $\A(\Gb)/ I \cong \A(H)$ with $I = \text{ker}(\pi_A)$. Since we know that $\A(H)\cong \A(\Gb)/ I$ is operator weakly amenable, the closed ideal $I$ has the trace extension property, i.e. any functional $\varphi \in I^*$ satisfying $\varphi(ab) = \varphi(ba)$, for any $a\in I$ and $b\in \A(\QG)$ can be extended to a tracial functional $\tau \in \A(\Gb)^*$. Here, we used an operator space version of \cite[Proposition 2.8.66 (iii)]{Dales}, whose proof is identical with the classical one. Moreover, we also know that $\A(H)$ is not weakly amenable, so that by \cite[Proposition 2.8.66 (iv)]{Dales} we can conclude that $\A(\Gb)$ is not weakly amenable.
\end{proof}

\begin{rem}
	\begin{enumerate}
		%\item Theorem \ref{thm-NWA} can be applied the case of $\Gb = O^+_N$, $N\ge 3$ and $\Gb = $ $q$-deformations of compact simply connected Lie groups $G_q$ and so on. However, the case of $O^+_2$ is excluded, so that Theorem \ref{thm-NWA-O^+_2} has a small advantage.
		
		\item It is not known whether $\A(O^+_N)$, $N\ge 2$, is operator weakly amenable or not (the derivations we construct here fail to be completely bounded if we pass to the context of `standard' Fourier algebras), whilst we know that $\A(O^+_N)$, $N\ge 3$ is not operator amenable thanks to the non-amenability of $\widehat{O^+_N}$, $N\ge 3$ (see \cite[Theorem 4.5]{Ruan2}).
		
		\item In the dual situation a complete characterization of operator amenability is known. More precisely, for a compact quantum group $\QG$, the $L^1$-algebra $L^1(\QG)$ is operator amenable if and only if $\QG$ is co-amenable (i.e. $\hQG$ is amenable) and of Kac type (\cite{CasLeeRic}).
		
		\item One might hope to use the existence of classical subgroups inside $\Gb$ to prove non-operator weak amenability of $\A_\Delta(\Gb)$, but the `$\A_{\Delta}$'-algebras do not have functorial property even for the classical groups. More precisely, for a quantum closed subgroup $\Hb$ of $\Gb$ the usual restriction homomorphism $\pi : \PolQG \to \Pol(\Hb)$ need not extend to $\pi_\Delta : \A_\Delta(\Gb) \to \A_\Delta(\Hb).$ See \cite[Section 3]{LSS} for the details.
			\end{enumerate}
\end{rem}

\end{document}